\documentclass[11pt, twoside]{article}
\usepackage[a4paper,width=150mm,top=25mm,bottom=30mm]{geometry}
\usepackage[all,cmtip]{xy}
\usepackage{comment}
\usepackage[pdftex]{pict2e}
\usepackage{amsmath,amssymb,amscd,amsthm}
\usepackage{a4wide,pstricks}
\usepackage{pdfsync}
\usepackage{fancyhdr}
\usepackage{makebox, bigstrut}
\usepackage[T1]{fontenc}
\usepackage[utf8]{inputenc} 
\usepackage[english]{babel} 
\usepackage{setspace}
\usepackage{bbm}
\usepackage{tasks}
\usepackage{dsfont} 
\usepackage{amsthm}
\usepackage{amsmath}
\usepackage{amssymb}
\usepackage{mathrsfs}
\usepackage{mathtools}
\usepackage{enumitem}
\usepackage{hyperref}
\usepackage{amssymb}
\usepackage{calligra}
\usepackage{ytableau}
\usepackage{marginnote}
\usepackage{cleveref}
\usepackage{babel}
\usepackage{bm}
\usepackage{tikz-cd}
\usepackage{faktor}
\renewcommand{\phi}{\varphi}

\let\emptyset\varnothing

\newcommand{\C}{\mathbb{C}}

\newcommand{\N}{\mathbb{N}}

\newcommand{\Q}{{\mathbb Q}}

\newcommand{\Z}{\mathbb{Z}}
\newcommand{\F}{\mathbb{F}}

\xyoption{all}
\input{xypic}

\newcommand{\Ch}{\mathcal{X}}

\theoremstyle{plain}
\numberwithin{equation}{subsection}
\crefformat{section}{\S#2#1#3} 
\crefformat{subsection}{\S#2#1#3}
\crefformat{subsubsection}{\S#2#1#3}
\let\oldmarginpar\marginpar
\renewcommand\marginpar[1]{\-\oldmarginpar[\raggedleft\footnotesize #1]
{\raggedright\footnotesize #1}}

\makeindex

\newtheorem{teorema}{Theorem}[subsection]
\newtheorem{prop}[teorema]{Proposition}

\newtheorem{lemma}[teorema]{Lemma}

\theoremstyle{remark}
\newtheorem{oss}[teorema]{Remark}
\newtheorem{esempio}[teorema]{Example}

\theoremstyle{definition}
\newtheorem{definizione}[teorema]{Definition}

\newcounter{margin}
{\end{itshape}  \bigskip}

\DeclareMathOperator{\Hom}{Hom}

\DeclareMathOperator{\Mat}{Mat}

\DeclareMathOperator{\Ker}{Ker}

\DeclareMathOperator{\Gl}{GL}
\DeclareMathOperator{\End}{End}
\DeclareMathOperator{\Imm}{Im}

\DeclareMathOperator{\Aut}{Aut}

\DeclareMathOperator{\tr}{tr}

\DeclareMathOperator{\ch}{char}

\DeclareMathOperator{\Stab}{Stab}

\DeclareMathOperator{\Rep}{Rep}

\DeclareMathOperator{\topp}{top}
\DeclareMathOperator{\Rad}{Rad}

\DeclareMathOperator{\Fr}{Fr}

\DeclareMathOperator{\Plexp}{Exp}
\DeclareMathOperator{\Plelog}{Log}

\DeclareMathOperator{\Ind}{Ind}

\DeclareMathOperator{\Coeff}{Coeff}

\DeclareMathOperator{\rank}{rank}
\DeclareMathOperator{\chh}{ch}
\DeclareMathOperator{\PGl}{PGL}
\begin{document}

\title{A generalization of Kac polynomials and tensor product \\ of  representations of $\Gl_n(\F_q)$}

\author{ Tommaso Scognamiglio
\\ {\it Université Paris Cité/IMJ-PRG}
\\{\tt tommaso.scognamiglio@imj-prg.fr}
}

\maketitle
\pagestyle{myheadings}

\begin{abstract}

Given a \emph{generic} $k$-tuple $(\mathcal{X}_1,\dots,\mathcal{X}_k)$ of split semisimple irreducible characters of ${\rm GL}_n(\F_q)$, Hausel, Letellier and Rodriguez-Villegas \cite[Theorem 1.4.1]{AH} constructed a \emph{star-shaped} quiver $Q=(I,\Omega)$ together with a  dimension vector $\alpha\in\mathbb{N}^I$ and they proved that 
\begin{equation}
\left\langle\mathcal{X}_1\otimes\cdots\otimes \mathcal{X}_k,1\right\rangle=a_{Q,\alpha}(q)
\label{eq}\end{equation}
where $a_{Q,\alpha}(t)\in\mathbb{Z}[t]$ is the so-called \emph{Kac polynomial}, i.e. it is the counting polynomial for the number of isomorphism classes of absolutely indecomposable representations of $Q$ of dimension vector $\alpha$ over finite fields. Moreover it was conjectured by Kac \cite{kacconj} and proved by Hausel-Letellier-Villegas \cite{aha3} that $a_{Q,\alpha}(t)$ has non-negative integer coefficients. From the above formula together with  Kac's results \cite{kacconj} they deduced that  $\left\langle\mathcal{X}_1\otimes\cdots\otimes \mathcal{X}_k,1\right\rangle\neq 0$ if and only if $\alpha$ is a root of $Q$ ; moreover $\left\langle\mathcal{X}_1\otimes\cdots\otimes \mathcal{X}_k,1\right\rangle=1$ exactly when $\alpha$ is a real root.
\bigskip

In this paper we extend their result to  any $k$-tuple $(\mathcal{X}_1,\dots,\mathcal{X}_k)$ of split semisimple irreducible characters  (which are not necessarily generic). To do that we introduce a stratification  indexed by subsets $V\subset\mathbb{N}^I$ on the set  of $k$-tuples of split semisimple irreducible characters of ${\rm GL}_n(\F_q)$ . The part corresponding to $V=\{\alpha\}$ consists of the subset of generic $k$-tuples $(\mathcal{X}_1,\dots,\mathcal{X}_k)$. A $k$-tuple $(\mathcal{X}_1,\dots,\mathcal{X}_k)$  in the stratum corresponding to $V\subset \mathbb{N}^I$ is said to be of level $V$.

A representation $\rho$ of $(Q,\alpha)$ is said to be of level at most $V\subset \mathbb{N}^I$ if the dimension vectors of the indecomposable components of $\rho\otimes_{\mathbb{F}_q}\overline{\mathbb{F}}_q$  belong to $V$. 

Given a  $k$-tuple $(\mathcal{X}_1,\dots,\mathcal{X}_k)$ of level $V$, our main theorem is the following generalisation of Formula (\ref{eq})

$$
 \left\langle\mathcal{X}_1\otimes\cdots\otimes \mathcal{X}_k,1\right\rangle=M_{Q,\alpha,V}(q)
 $$
where $M_{Q,\alpha,V}(t)\in\mathbb{Z}[t]$ is the counting polynomial for the number of isomorphism classes of representations of $(Q,\alpha)$ over $\mathbb{F}_q$ of level at most $V$. Moreover we prove a formula expressing $M_{Q,\alpha,V}(t)$ in terms of  Kac polynomials and so we get a formula expressing any  multiplicity $\left\langle\mathcal{X}_1\otimes\cdots\otimes \mathcal{X}_k,1\right\rangle$ in terms of  generic ones. As another consequence we prove that $ \left\langle\mathcal{X}_1\otimes\cdots\otimes \mathcal{X}_k,1\right\rangle$ is a polynomial in $q$ with non-negative integer coefficients and we give a criterion for its non vanishing  in terms of the root system of $Q$.

\end{abstract}

\tableofcontents
\newpage

\section{Introduction}

The aim of this paper is to study the relationship between the computation of multiplicities for tensor product of representations of $\Gl_n(\F_q)$ and quiver representations.

\subsection{Review on multiplicities}

The character table of $\Gl_n(\F_q)$ is known since 1955 by the work of Green  \cite{green}, who gave a combinatorial  description of it. His formulae for the values of the irreducible characters have an algorithmic nature.

Deligne and Lusztig \cite{DLu} later introduced $\ell$-adic cohomological methods to the study of the representation theory of finite reductive groups. Using this approach, in \cite{CharLu} Lusztig found a geometric way to construct the irreducible characters of a finite reductive group. In the same book, he introduced the notion of a semisimple and unipotent irreducible character by analogy with the Jordan decomposition for the conjugacy classes.

For the finite general linear group $\Gl_n(\F_q)$, Lusztig's costruction led to a geometric  interpretation of the character table found by Green, see  for example Lusztig and Srinivasan \cite{LSr}. 

\vspace{8 pt} 

Given  $\Ch_1,\Ch_2,\Ch_3$ irreducible characters of $\Gl_n(\F_q)$, the multiplicity $\left<\Ch_1 \otimes \Ch_2 ,\Ch_3 \right>$ is given by the formula \begin{equation}
    \label{boh}
    \left<\Ch_1 \otimes \Ch_2 ,\Ch_3 \right>=\dfrac{1}{|\Gl_n(\F_q)|}\sum_{g \in \Gl_n(q)}\Ch_1(g)\Ch_2(g)\overline{\Ch_3(g)}
\end{equation}
Altough the character table of $\Gl_n(\F_q)$ is known for a long time, it is not easy to extract general information  from  Formula (\ref{boh}) above, due to the inductive description of the values of the  characters. 

\vspace{8 pt}

\begin{esempio}
\label{unip}
Recall that the \textit{unipotent} characters of $\Gl_n(\F_q)$, which are the "building blocks" of the character table, are in bijection with the irreducible representations of $S_n$ and so with the partitions of $n$.

For a partition $\mu$, we denote by $\chi_{\mu}$ the associated character of $S_n$ and by $\mathcal{X}_{\mu}$ the associated unipotent character of $\Gl_n(\F_q)$ (in our paramatrization, we associate to the partition $(n)$ the trivial character $1$).

 From  Formula (\ref{boh}), it is nearly impossible to  obtain directly a  combinatorial description of the set $\{(\lambda,\mu,\nu) \in \mathcal{P}_{n} \ | \ \langle \Ch_{\lambda} \otimes \Ch_{\mu},\Ch_{\nu} \rangle \neq 0 \}$, where $\mathcal{P}_n$ is the set of the partitions of $n$.  

  Already for $S_n$, the problem of giving a  combinatorial criterion for the non-vanishing of the \textit{Kronecker coefficients} $g_{\lambda,\mu}^{\nu}\coloneqq \langle \chi_{\mu} \otimes \chi_{\lambda},\chi_{\nu}\rangle$, is still open and is a very active area of research.
  
  Interestingly, the two problems were shown to be related by Letellier \cite{Unip}: in particular,  Letellier \cite[Proposition 1.2.4]{Unip},  showed that if $g_{\lambda,\mu}^{\nu}\neq 0$ then $\langle \Ch_{\lambda} \otimes \Ch_{\mu},\Ch_{\nu} \rangle \neq 0$ too.
\end{esempio} 
\vspace{10 pt}

Recall that the multiplicity $\langle \Ch_1 \otimes \Ch_2,\Ch_3 \rangle$ is equal to $\langle \Ch_1 \otimes \Ch_2 \otimes \Ch_3^*,1 \rangle$ where $\Ch_3^*$ is the dual character of $\Ch_3$. One of the aims of this paper is to contribute to the study of the multiplicites $\left<\Ch_1 \otimes \cdots \otimes \Ch_k,1\right>$ for any $k$-tuple of irreducible characters $(\Ch_1,\dots ,\Ch_k)$. 

The understanding of these quantites is still an open problem in general but substantial progress were made recently.

The first cases studied in the literature concerned $k$-tuples $(\Ch_1,\dots,\Ch_k)$ where each $\Ch_i$ is an unipotent character. Hiss, L\"ubeck and Mattig \cite{Hiss} computed, for example, the multiplicities $\langle \Ch_1 \otimes \Ch_2 \otimes \Ch_3,1 \rangle$ for unipotent characters $\Ch_1,\Ch_2,\Ch_3$ and $n \leq 8$ using CHEVIE. They noticed that these quantities are polynomials in $q$, with positive coefficients. Lusztig studied multiplicities for unipotent character sheaves of $\PGl_2$ \cite{lusztignotes}. 

The first general results  were obtained in the papers \cite[Theorem 1.4.1]{HA},\cite[Theorem 3.2.7]{AH} by Hausel, Letellier, Rodriguez-Villegas and were later generalised by Letellier \cite[Theorem 6.10.1,Theorem 7.4.1]{letellier2},\cite[Theorem 3.3.1, Proposition 3.4.1]{Unip}. 

The authors \cite{HA},\cite{AH},\cite{letellier2} restricted themselves to a certain class of $k$-tuples $(\Ch_1,\dots,\Ch_k)$, called \textit{generic} (see  Definition \ref{generic1}). Notice that a $k$-tuple of unipotent characters is  never generic.

For generic $k$-tuples, the calculations of  Formula (\ref{boh}) simplify and we can deduce a general combinatorial formula for the multiplicity $\langle \Ch_1 \otimes \cdots \otimes \Ch_k,1 \rangle$, involving symmetric functions (see \cite[Theorem 6.10.1]{letellier2}).

The generic case is  particularly interesting due to the surprising connection with quiver representations  and Kac polynomials. In \cite[Corollary 3.4.2]{AH}\cite[Theorem 7.4.5]{letellier2}, the authors found, for example, a criterion for the non-vanishing of the multiplicity $\langle \Ch_1 \otimes \cdots \otimes \Ch_k,1 \rangle$ for generic $k$-tuples in terms of the root system of a certain star-shaped quiver.

In the same papers \cite[Theorem 6.2.1]{HA},\cite[Theorem 7.4.1]{letellier2}, it is given also a geometric interpretation of the quantities $\langle \Ch_1 \otimes \cdots \otimes \Ch_k,1 \rangle$ in terms of the cohomology of certain quiver varieties.

\vspace{8 pt}

In this paper we are interested in the study of the multiplicities  $\langle \Ch_1 \otimes \cdots \otimes \Ch_k,1 \rangle$ by dropping the genericity constraint. We will mainly focus on $k$-tuples $(\Ch_1,\dots,\Ch_k)$ where each $\Ch_i$ is a semisimple (split) character (see  \cref{boh2}).

\subsection{Generalities on quivers and Kac polynomials}
\label{questarevisionenonfiniscemai}
Let $\Gamma=(J,\Omega)$ be a (finite) quiver, where $J$ is its set of vertices and $\Omega$ its set of arrows. In \cite{kacconj}, for each dimension vector $\beta \in \N^J$, Kac introduced a polynomial with integer coefficients $a_{\Gamma,\beta}(t)$, called \textit{Kac polynomial}, defined by the fact that $a_{\Gamma,\beta}(q)$ counts the number of isomorphism classes of \textit{absolutely indecomposable} representations of $\Gamma$ (see Definiton \ref{abs}) of dimension $\beta$ over $\mathbb{F}_q$, for any $q$.

Kac showed that $a_{\Gamma,\beta}(t)$ is non-zero if and only if $\beta$ is a root of $Q$ and conjectured that it has non-negative coefficients. 

The latter conjecture was first proved  by Crawley-Boevey and Van der Bergh \cite{crawley-boevey-etal} in the case of indivisible $\beta$ (i.e $\gcd(\beta_j)_{j \in J}=1$)
and later for any $\beta$ by Hausel, Letellier, Rodriguez-Villegas in \cite{aha3}. In both cases, the authors obtained the non-negativity property by giving a description of the coefficients in terms of the cohomology of certain quiver varieties.

For instance, if  $\beta$ is indivisible, in \cite[End of Proof 2.4]{crawley-boevey-etal}, it is shown, that  there  is an equality $$P_c(\mathcal{Q},t)=t^{d_{\mathcal{Q}}}a_{Q,\alpha}(t^2) ,$$ for a certain quiver variety $\mathcal{Q}$, where $P_c(\mathcal{Q},t)$ is the compactly supported Poincar\'e polynomial of $\mathcal{Q}$ and $d_{\mathcal{Q}}$ is the dimension of $\mathcal{Q}$.

Kac polynomials enjoy other remarkable properties. For instance, they are related to the representation theory of Kac-Moody Lie algebras (see \cite{kacconj}) and they appear in the theory of Donaldson-Thomas invariants of the quiver $\Gamma$ (see \cite{DT},\cite{kontsevich_stability}).

\subsection{Quivers and characters} 
\label{boh2}
Let $L$ be the Levi subgroup $L=\Gl_{m_1}(\F_q) \times \cdots \times \Gl_{m_s}(\F_q)$ embedded block diagonally in $\Gl_n(\F_q)$, where $m_1,\dots,m_s$ are nonnegative integers such that $m_1+\dots+ m_s=n$. Fix a linear character $\gamma:L \to \C^*$ given by $\gamma(M_1,\dots,M_s)=\gamma_1(\det(M_1))\cdots \gamma_s(\det(M_s))$ for $\gamma_1,\dots,\gamma_s \in \Hom(\mathbb{F}_q^*,\C^*)$.

We denote by $R^G_{L}(\gamma)$ the Harisha-Chandra induced character of $\Gl_n(\F_q)$. Recall that if $\gamma_i \neq \gamma_j$ for each $i \neq j$ the character $R^G_L(\gamma)$ is irreducible. The irreducible characters of this form are called \textit{semisimple split} (see \cref{splitdef} for more details). 

Consider now a $k$-tuple of semisimple split characters $(R^G_{L_1}(\delta_1),\dots,R^G_{L_k}(\delta_k))$ , where, for $i=1,\dots,k$, we have $L_i=\Gl_{m_{i,1}}(\F_q) \times \cdots \times \Gl_{m_{i,s_i}}(\F_q)$ and $\delta_i(M_1,\dots,M_{s_i})=\delta_{i,1}(\det(M_1))\cdots$ $\delta_{i,s_i}(M_{s_i})$. To such a $k$-tuple we associate the following star-shaped quiver $Q$.

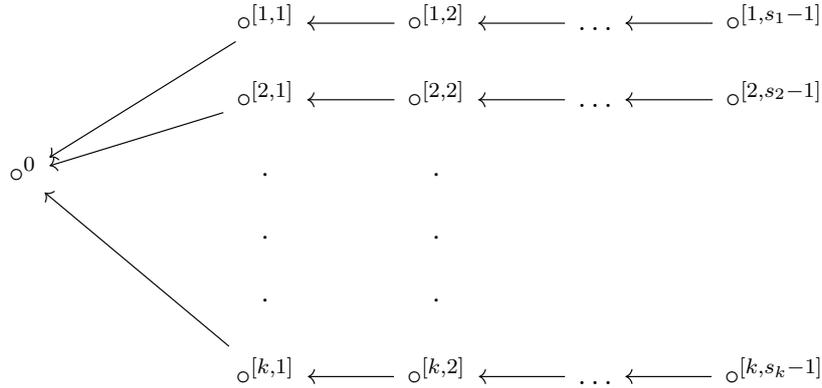
\vspace{8 pt}

\begin{center}
    \begin{tikzcd}[row sep=1em,column sep=3em]
    & &\circ^{[1,1]} \arrow[ddll,""] &\circ^{[1,2]} \arrow[l,""]  &\dots \arrow[l,""] &\circ^{[1,s_1-1]} \arrow[l,""]\\
    & &\circ^{[2,1]} \arrow[dll,""] &\circ^{[2,2]} \arrow[l,""] &\dots \arrow[l,""] &\circ^{[2,s_2-1]} \arrow[l,""]\\
    \circ^0  & &\cdot &\cdot\\
    & &\cdot &\cdot\\
    & &\cdot &\cdot\\
    & &\circ^{[k,1]} \arrow[uuull,""] &\circ^{[k,2]} \arrow[l,""]  &\dots \arrow[l,""] &\circ^{[k,s_k-1]} \arrow[l,""]
    \end{tikzcd}
\end{center}

\vspace{8 pt}

Let $I$ be the set of vertices of $Q$ and let $\alpha$ be the dimension vector $\alpha \in \N^I$ defined as $\alpha_0=n$ and $\alpha_{[i,j]}=n-\sum_{h=1}^jm_{i,j}$. Notice that the quiver $Q$ and the vector $\alpha$ depend only on the Levi subgroups $L_1,\dots,L_k$ and not on the characters $\delta_1,\dots,\delta_k$. 

\vspace{8 pt}

\vspace{8 pt}

In \cite[Theorem 3.2.7]{AH} Hausel, Letellier, Rodriguez-Villegas showed  that if the $\delta_i$'s are chosen so that $(R^G_{L_i}(\delta_i))_{i=1}^k$ is generic (such a choice is always possible if $q$ is big enough), there is an equality: \begin{equation}
     \label{introd1}
     \langle  R^G_{L_{1}}(\delta_1)\otimes \cdots \otimes R^G_{L_{k}}(\delta_k),1 \rangle=a_{Q,\alpha}(q).
 \end{equation}
 
 Notice that the Formula (\ref{introd1}) implies that, as long as the choice of the $\delta_i$'s is generic, the multiplicity $\langle  R^G_{L_{1}}(\delta_1)\otimes \cdots \otimes R^G_{L_{k}}(\delta_k),1 \rangle$ does not depend on the characters $\delta_i$'s: this is not clear a priori from the Formula (\ref{boh}). As a consequence of Formula (\ref{introd1}),  the multiplicity $\langle  R^G_{L_{1}}(\delta_1)\otimes \cdots \otimes R^G_{L_{k}}(\delta_k),1 \rangle$ is non-zero if and only if $\alpha \in \Phi^+(Q)$.

\vspace{14 pt} 
 
The aim of this paper is to generalize these results  to any $k$-tuple $(R^G_{L_i}(\delta_i))_{i=1}^k$. The main results of this paper  will remain valid without  assuming that the characters $R^G_{L_1}(\delta_1),\dots, R^G_{L_k}(\delta_k)$ are irreducible.  

\vspace{8 pt}

To study the non-generic case, we will begin by defining a partition both on the set of $k$-tuples of semisimple split characters $(R^G_{L_1}(\gamma_1),\dots,R^G_{L_k}(\gamma_k))$ and on the set of representations of $Q$ of dimension $\alpha$, indexed by subsets $V \subseteq \N^I$. 

The  level of the partition associated to $V=\{\alpha\}$  will correspond to the case of generic $k$-tuples/absolutely indecomposable representations  respectively. 

\vspace{10 pt}
Consider more generally any finite quiver $\Gamma=(J,\Omega)$. To a representation $M$ of $\Gamma$ over $\F_q$, we associate the following subset $\mathcal{H}_M \subseteq \N^I$. Given the decomposition into indecomposable components $$M \otimes_{\F_q} \overline{\F_q} \cong  M_1 \oplus \cdots \oplus M_h ,$$ we define $$\mathcal{H}_M \coloneqq \{ 0 < \beta \leq \dim M \ | \  \exists 1 \leq j \leq h \text{ s.t } \beta=\dim M_j\} .$$

For any $V \subseteq \N^J$, we give the following   definition of the representations of $\Gamma$ of level $V$ over $\F_q$ (see \ref{W-generic}).

\begin{definizione}
\label{definitionlevelvintrod}
A representation $M$ of $\Gamma$ over $\F_q$ is said to be of level $V$ if $\mathcal{H}_M=\{0 < \beta \leq \dim M \ | \ \beta \in V \}$.
\end{definizione}

Notice that, for $v \in \N^J$, a  representation of dimension  $v$ is of level $\{v\}$ if and only if it is absolutely indecomposable . In particular, the number of isomorphism classes of representations of level $\{v\}$ and dimension $v$ over finite fields is counted by the Kac polynomial $a_{\Gamma,v}(t)$.  
\vspace{10 pt}

However, for a general $V \subseteq \N^J$, the counting of the isomorphism classes of the representations of $\Gamma$ of prescribed dimension and of level $V$ over $\F_q$ does not seem to give an interesting generalization of Kac polynomials. In this direction, to obtain such a generalization, we introduce the following definition of a representation of level at most $V$ (see Definition \ref{representationsoflevelatmostV}).  

\begin{definizione}
\label{representationsoflevelatmostVintrod}
For a subset $V \subseteq \N^J$, a representation $M$ is said to be of level at most $V$ if it is of level $V'$ for some $V' \subseteq V$, i.e if and only if $\mathcal{H}_M  \subseteq V$. 
\end{definizione}

For any $\beta \in \N^J$ and any $V \subseteq \N^J$,
we show  that the number of isomorphism classes of  representations of $\Gamma$ of level at most $V$  of dimension $\beta$ over $\F_q$ is equal to the evaluation of a polynomial $M_{\Gamma,\beta,V}(t) \in \Z[t]$ at $t=q$. 

Moreover, we prove a formula for the generating function of the polynomials $M_{\Gamma,\beta,V}(t)$ of the following type (see Lemma \ref{formulaV}): 
\begin{equation}
\label{introd2}
    \Plexp\left(\sum_{\gamma \in V}a_{\Gamma,\gamma}(t)y^{\gamma}\right)=\sum_{\beta \in \N^I}M_{\Gamma,\beta,V}(t)y^{\beta}.\end{equation}
where $\Plexp$ is the plethystic exponential.

\vspace{12 pt}
Consider now the quiver $Q$ and the dimension vector $\alpha$ introduced above. Let $(\N^I)^* \subseteq \N^I$ be the subset of vectors with non-increasing coordinates along the legs.

In \cref{parameter}, to any $k$-tuple $\mathcal{X}=(R^G_{L_1}(\delta_1),\dots,R^G_{L_k}(\delta_k))$ we associate an element $\sigma_{\mathcal{X}} \in \Hom(\F_q^*,\C^*)^I$ (see Formula (\ref{definitionsigma})). Let $\mathcal{H}^*_{\sigma_{\mathcal{X}},\alpha} \subseteq (\N^I)^*$ be the subset defined by $$\mathcal{H}^*_{\sigma_{\mathcal{X}},\alpha} \coloneqq \{0 <\beta \leq \alpha \ | \ \sigma_{\mathcal{X}}^{\beta}=1\} $$ where $\displaystyle\sigma_{\mathcal{X}}^{\beta}\coloneqq \prod_{i \in I}(\sigma_{\mathcal{X}})_i^{\beta_i}$. For any $V \subseteq (\N^I)^*$, we give the following definition of a $k$-tuple $(R^G_{L_i}(\delta_i))_{i=1}^k$  of level $V$ (see \ref{defgeneric}).

\begin{definizione}
A $k$-tuple $\mathcal{X}=(R^G_{L_1}(\delta_1),\dots,R^G_{L_k}(\delta_k))$ is said to be of level $V$ if $\mathcal{H}^*_{\sigma_{\mathcal{X}},\alpha}=\{0 <\beta \leq \alpha \ | \ \beta \in V \}$.
\end{definizione}

For $V=\{\alpha\}$, we show that if a  $k$-tuple $(R^G_{L_1}(\delta_1),\dots,R^G_{L_k}(\delta_k))$ is of level $\{\alpha\}$ then it is generic (see Proposition \ref{generic2}).

\vspace{8 pt}

The main result of this paper extends Formula (\ref{introd1}) by relating the multiplicity for $k$-tuples of level $V$ and representations of level at most $V$ in the following way.

\begin{teorema}
\label{mainteo2}
Let $V \subseteq (\N^I)^*$ and $\mathcal{X}=(R^G_{L_{1}}(\delta_1),\dots ,R^G_{L_{k}}(\delta_k))$ be a $k$-tuple of level $V$. The following equality holds: 
\begin{equation}
    \label{mainteo3}
    \langle R^G_{L_{1}}(\delta_1)\otimes \cdots \otimes R^G_{L_{k}}(\delta_k),1 \rangle=M_{Q,\alpha,V}(q)
\end{equation}
\end{teorema}

\vspace{8 pt}

From Theorem \ref{mainteo2} and  Formula (\ref{introd2}), in Proposition \ref{nonvan} we obtain the following criterion  for the non-vanishing of the multiplicity $\langle R^G_{L_1}(\delta_1) \otimes \cdots \otimes R^G_{L_k}(\delta_k),1 \rangle$, generalizing the criterion for generic $k$-tuples of \cite[Corollary 1.4.2]{AH} .

\begin{prop}
For a $k$-tuple $(R^G_{L_1}(\delta_1),\dots,R^G_{L_k}(\delta_k))$ of level $V$, the multiplicity $$\left< R^G_{L_1}(\delta_1) \otimes \cdots \otimes R^G_{L_k}(\delta_k) ,1\right>$$ is non-zero if and only there exist
\begin{itemize}
    \item $\beta_1,\dots, \beta_r \in \Phi^+(Q) \cap V$
    \item $m_1,\dots ,m_r \in \N$
\end{itemize}
such that $m_1\beta_1+\cdots +m_r\beta_r=\alpha$
\end{prop}

\vspace{10 pt}

\begin{oss}

One of the interesting consequence of Theorem \ref{mainteo2} and Formula (\ref{introd2}) is that they provide a way to express the multiplicity $\langle R^G_{L_1}(\delta_1) \otimes \cdots \otimes R^G_{L_k}(\delta_k),1 \rangle$ for any semisimple split $k$-tuple $(R^G_{L_1}(\delta_1),\dots,R^G_{L_k}(\delta_k))$ in terms of the multiplicities for certain \textit{generic} ones, for which the articles \cite{HA}, \cite{AH}, \cite{letellier2} provide a thorough description.

 The same kind of result, i.e finding a way to express  what happens for the non-generic cases in terms of the generic ones, has already appeared in different works related to this subject.
 
 Davison \cite[Theorem B]{Bro}, for example, recently showed that the cohomology of non-generic quiver stacks can be expressed in terms of Kac polynomials.

 In \cite[Proposition 1.2.1]{Unip}, Letellier proved a formula which expresses the multiplicities $\langle \Ch_1 \otimes \cdots \otimes \Ch_k,1 \rangle$ for $k$-tuples of \textit{unipotent} characters $(\Ch_1,\dots,\Ch_k)$ in terms of the multiplicities for generic $k$-tuples of  twisted unipotent characters. 
\end{oss}

\vspace{8 pt}
 
\begin{oss}
Let $g \in \N$ and let $\Lambda$ be the character of $\Gl_n(\F_q)$ obtained by its conjugation action on $\mathfrak{gl}_n(\F_q)^g$. Theorem \ref{mainteo} shows more generally that, for a $k$-tuple $(R^G_{L_i}(\gamma_i))_{i=1}^k$ of level $V$, there is an equality $\langle \Lambda \otimes R^G_{L_{1}}(\gamma_1)\otimes \cdots \otimes R^G_{L_{k}}(\gamma_k),1 \rangle=M_{Q',\alpha,V}(q)$ where $Q'$ is the quiver obtained from $Q$ by adding $g$ loops at the central vertex $0$. Equation (\ref{mainteo3}) is thus the $g=0$ case of this result.

The presence of the character $\Lambda$ seems not very relevant from the point of view of the representation theory of $\Gl_n(\F_q)$. This generalization is however interesting from the point of view  of quiver representations and for the relationship of the formula for the multiplicity $\langle \Lambda \otimes R^G_{L_{1}}(\gamma_1)\otimes \cdots \otimes R^G_{L_{k}}(\gamma_k),1 \rangle$ to those computing the E-polynomials of character varieties for a Riemann surface of genus $g$ (see for example \cite[Theorem 5.2.1,Theorem 6.1.1]{HA}) 
\end{oss}

\subsection{Final questions}

\subsubsection{Geometric interpretation for representations of level $V$}

Consider a finite quiver $\Gamma=(J,\Omega)$ and  a dimension vector $\eta \in \N^J$. Notice that for $V=\{\eta\}$, there is an equality $M_{\Gamma,\eta,\{\eta\}}(t)=a_{\Gamma,\eta}(t)$. In this case $M_{\Gamma,\eta,V}(t)$ has thus non-negative coefficients which have an interpretation in terms of the cohomology of quiver varieties. 

From  Formula (\ref{introd2}) and the positivity of Kac polynomials, in Proposition \ref{nonnega1} we manage to show that, for any $V \subseteq \N^I,\eta \in \N^I$, the polynomial $M_{\Gamma,\eta,V}(t)$ has nonnegative integer coefficients. It would be interesting to look for  geometric interpretations of the coefficients of $M_{\Gamma,\eta,V}(t)$ for any $V$.

It is likely that the arguments of \cite[Theorem C]{Davison} could be used to relate the polynomials $M_{\Gamma,\eta,V}(t)$ to the DT-invariants of the Harder-Narasimhan part of the preprojective stack of $\Gamma$ whose graded associated has semistable factors with dimension vector belonging to $V$. 

However, it seems that there is no immediate way to relate the polynomial $M_{\Gamma,\eta,V}(t)$ to the geometry of quiver varieties for a general $V$. For instance, the argument used by Crawley-Boevey and Van der Bergh \cite[Proposition 2.2.1]{crawley-boevey-etal} to link absolutely indecomposable representations to the cohomology of quiver varieties cannot be extended to an arbitrary level $V$.

\subsubsection{Multiplicity for a general $k$-tuple of characters}
It remains an open question to obtain a full understanding of the multiplicity $\langle \Ch_1 \otimes \cdots \otimes \Ch_k,1 \rangle$ for an arbitrary $k$-tuple $\mathcal{X}=(\Ch_1,\dots,\Ch_k)$ of not necessarily split semisimple irreducible characters.  

However, we do not expect that it would be possible to obtain an immediate generalization of Formula (\ref{mainteo3}) or a direct interpretation of the quantity $\langle \Ch_1 \otimes \cdots \otimes \Ch_k,1 \rangle$ in terms of quiver representations.

\vspace{8 pt}

In the generic case, Letellier  \cite[Theorem 6.10.1]{letellier2} found a  formula
giving a combinatorial expression for the multiplicity $\langle \Ch_1 \otimes \cdots \otimes \Ch_k,1 \rangle$ for any generic $k$-tuple $\mathcal{X}$. His formula has an interpretation in terms of the cohomology of a certain quiver variety (see \cite{letelliervillegas}). However, Letellier's result does not provide a straightforward way to relate the multiplicity $\langle \Ch_1 \otimes \cdots \otimes \Ch_k,1 \rangle$ to the counting of absolutely indecomposable representations of star-shaped quivers.  

\paragraph{Acknowledgements.}

I would like to warmly thank Emmanuel Letellier for bringing this subject to my attention and many useful suggestions.

I would also like to thank the anonymous referee for many useful suggestions about a first draft of the paper.

\section{Preliminaries}

\subsection{Plethystic operations}
\label{plety}
Let $I$ be a finite set and let $R$ be the ring of power series in $|I|$-variables, with coefficients in $\Q(t)$ i.e $R\coloneqq \Q(t)[[y_i]]_{i \in I}$. For $\alpha \in \N^I$, we will denote by $y^{\alpha}$ the monomial $\displaystyle\prod_{i \in I}y_i^{\alpha_i}$ and we will  denote by $R^+ \subseteq R$ the ideal generated by the $y_i$'s for $i \in I$. 

The ring $R$ is endowed with a natural $\lambda$-ring structure  (for basic constructions and definition about $\lambda$-rings see for example \cite{getzler_mixed}), given by the following Adam operations $$\psi_d(f(t,y))\coloneqq f(t^d,y^d) .$$

 We define the \textit{plethystic exponential}  $\Plexp:R^+ \to 1+R^+$ by the following rule \begin{equation}\label{exp1}\Plexp(f) \coloneqq \exp\left(\sum_{n \geq 1}\frac{\psi_n(f)}{n}\right) .\end{equation}

Notice that for  $A,B \in R^+$  we have  $\Plexp(A+B)=\Plexp(A)\Plexp(B) .$

\begin{esempio}

If $R=\Q(t)[[T]]$ (i.e $|I|=1$), we have: \begin{equation}
    \label{esempio}
    \Plexp(T)=\exp\left(\sum_{n \geq 1} \frac{T^n}{n}\right)=\exp\left(\log\left(\frac{1}{1-T}\right)\right)=\frac{1}{1-T}
\end{equation}

\end{esempio}

\vspace{10 pt}

The plethystic exponential admits an inverse operation  $\Plelog:1+R^+ \to R^+$ known as \textit{plethystic logarithm}. The plethystic logarithm can either be defined by the property $\Plelog\Plexp(f)=f $ or by the following explicit rule. For $\alpha \in \N^I$ we put $$\overline{\alpha}\coloneqq \gcd(\alpha_i)_{i \in I}$$ and we define $U_{\alpha} \in \Q(t)$ as the coefficients appearing in the expansion: \begin{equation}
\log(f)=\sum_{\alpha \in \N^I}\frac{U_{\alpha}}{\overline{\alpha}}y^{\alpha}.
\end{equation}
Then put
 \begin{equation}
    \label{log}
    \Plelog(f)\coloneqq\sum_{\alpha \in \N^I}V_{\alpha}y^{\alpha}
\end{equation}
where \begin{equation}
    V_{\alpha}\coloneqq\dfrac{1}{\overline{\alpha}}\sum_{d| \overline{\alpha}}\mu(d)U_{\frac{\alpha}{d}}(t^d)
\end{equation}

\vspace{10 pt}

Notice that for $A,B \in 1+R^+$, the following equality holds $\Plelog(AB)=\Plelog(A)+\Plelog(B) .$

\vspace{10 pt}

The plethystic exponential is well defined over the subring $R_{\Z}\coloneqq \Z(t)[[y_i]]_{i \in I}$  i.e:
\begin{lemma}
\label{Z}
Given $f \in \Z(t)[[y_i]]^+_{i \in I}$ we have $\Plexp(f) \in 1+ \Z(t)[[y_i]]_{i \in I}^+$
\end{lemma}

We recall also the useful lemma below (for a proof see \cite[Lemma 22]{mozgovoy}). For $g \in R$ let $$g_n\coloneqq \dfrac{1}{n}\sum_{s | n}\mu\left(\dfrac{n}{s}\right)\psi_s(g) .$$

\begin{lemma}
\label{exp}
Given $f_1,f_2 \in 1+R^+$ such that \begin{equation}
    \log(f_1)=\sum_{d \geq 1}g_d \log(\psi_d(f_2))
\end{equation}
the following equality holds
\begin{equation}
    \Plelog(f_1)=g \Plelog(f_2).
\end{equation}
\end{lemma}

\subsection{Quiver representations}

A quiver $Q$ is an oriented graph $Q=(I,\Omega)$, where $I$ is its set of vertices and  $\Omega$ is its set of arrows. We will always assume that $I,\Omega$ are finite sets. For an arrow $a:i \to j$  in $\Omega$ we denote by $i=t(a)$ its \textit{tail} and by $j=h(a)$ its \textit{head}. Fix a field $K$. A representation $M$ of $Q$ over $K$ is given by a (finite dimensional) $K$-vector space $V_i$ for each vertex $i \in I$ and by linear  maps $M_a:M_{t(a)} \to M_{h(a)}$ for each $a \in \Omega$. 

Given two representations $M,M'$ of $Q$, a morphism $f:M \to M'$  is given by a collection of linear maps $f_i:M_i \to M'_i$ such that, for all $a \in \Omega $, the following equality holds:  $f_{h(a)} M_a=M'_a f_{t(a)} $.
The category of representations of $Q$ over $K$ is denoted by $\Rep_K(Q)$. For a representation $M$, the \textit{dimension vector} $\dim M\in \N^I$ is the vector $\dim M\coloneqq(\dim M_i)_{i \in I}  $. It is an isomorphism invariant of the category $\Rep_k(Q)$. 

For a representation $M$ of dimension $\alpha$, up to fixing a basis of the vector spaces $M_i$ for each $i \in I$, we can assume that $M_i=K^{\alpha_i}$. For $a \in \Omega$, the linear map $M_{a}:K^{t(a)} \to K^{h(a)}$ can be therefore identified with a matrix in $\Mat(\alpha_{h(a)},\alpha_{t(a)},K)$. 

Consider then the affine space $$R(Q,\alpha)\coloneqq \bigoplus_{a \in \Omega} \Mat(\alpha_{h(a)},\alpha_{t(a)},K) .$$
 We can endow $R(Q,\alpha)$ with the action of the group $\Gl_{\alpha}=\displaystyle \prod_{i \in I}\Gl_{\alpha_i}$  defined by  $$g \cdot (M_a)_{a \in \Omega}=(g_{h(a)}M_ag_{t(a)}^{-1})_{a \in \Omega} .$$ The orbits of this action are exactly the isomorphism classes of representations of $Q$ of $\dim=\alpha$.

\vspace{10 pt}

 We briefly recall also the definition of the \textit{moment map} of $Q$. Denote by $\overline{Q}$ the double quiver $\overline{Q}=(I,\overline{\Omega})$ with the same vertices of $Q$ and as set of arrows $\overline{\Omega}=\{a,a^* \ | \ a \in \Omega \}$ where  $a^*:j \to i$ for $a:i \to j$. For $\alpha \in \N^I$,  the moment map $\mu_{\alpha}$ is the morphism $$\mu_{\alpha}:R(\overline{Q},\alpha) \to \End_0(\alpha) $$ $$(x_a,x_{a^*})_{a \in \Omega} \rightarrow \sum_{a \in \Omega }[x_a,x_{a^*}] $$ where $$ \End_0(\alpha)=\{(M_i) \in \End(\alpha) \ | \ \sum_{i \in I} \tr(M_i)=0 \} .$$ 

Given $\lambda \in K^I$ such that $\lambda \cdot \alpha=0$, the element $(\lambda_iI_{\alpha_i})_{i \in I}$ (which we still denote by $\lambda$) is a central element of $\End_0(\alpha)$ and the fiber $\mu_{\alpha}^{-1}(\lambda)$ is $\Gl_{\alpha}$ invariant. We denote by  $\mathcal{Q}_{\lambda,\alpha} $ the quiver variety associated to $\lambda$, which is defined as the GIT quotient $$\mathcal{Q}_{\lambda,\alpha}\coloneqq \mu^{-1}_{\alpha}(\lambda)//\Gl_{\alpha}   .$$

The quiver stack $\mathcal{M}_{\lambda,\alpha}$ is defined  as the quotient stack $$\mathcal{M}_{\lambda,\alpha}\coloneqq [\mu^{-1}_{\alpha}(\lambda)/\Gl_{\alpha}] .$$

\subsubsection{Krull-Schmidt decomposition and endomorphism rings}

The category $\Rep_K(Q)$ is  abelian and Krull-Schimdt, i.e i.e we have the following Theorem, see for example \cite[Theorem 1.11]{kirillov}. 

\begin{teorema}
Each object $M \in \Rep_K(Q)$ admits a decomposition into a direct sum of  indecomposable ones $$M=\bigoplus_{j \in J}M_j^{n_j} $$ and such a decomposition is unique up to permuting the factors $M_j$. 
\end{teorema}

Recall that a representation $M$ is indecomposable if and only if $\End(M)$ is a local algebra. 

Its maximal ideal is denoted by $\Rad(M) \subseteq \End(M)$. It is the set of nilpotent endomorphisms of $M$. The quotient $\End(M)/\Rad(M)$ is thus a division algebra, which is usually denoted by $\topp(M)$.

More generally, given two representations $M,N$ it is possible to define a subset  $\Rad(M,N) $ of $\Hom(M,N)$ as $$\Rad(M,N)=\{g \in \Hom(M,N) \ | \ 1+gf \in \Aut(N) \ , \forall f \in \Hom(N,M)\} .$$ If $M=N$ and $M$ is indecomposable,  $\Rad(M)=\Rad(M,M)$. In general, $\Rad(N,N)$ is still an ideal of $\End(N)$ and if $M,N$ are  non-isomorphic indecomposable representations, $\Rad(M,N)=\Hom(M,N)$. The radical is additive i.e $\Rad(M \oplus M',N)=\Rad(M,N) \oplus \Rad(M',N) .$ The following proposition holds (see \cite[Section 3.2]{repalgebras}): 

\begin{prop}
\label{rad}
Given a representation $X$  of $Q$ and an endomorphism $\phi \in \End(X)$, $\phi$ is invertible if and only if its class $\overline{\phi}$ in $\End(X)/\Rad(X)$ is invertible.
\end{prop}

\begin{oss}
\label{isomff}
Fix a representation $X$ and its Krull-Schmidt decomposition $X=\bigoplus_{j \in J}X_j^{r_j}$. For an endomorphism $\phi \in \End(X)$, we denote by $\phi_j$ the associated element in $\End(X_j^{r_j})$ and by $\overline{\phi}_j$ the associated element in $\End(X^{r_j}_j)/\Rad(X^{r_j}_j) \cong \Mat(r_j,\topp(X_j))$. As $\Rad(X_i,X_j)=\Hom(X_i,X_j)$ for every $i \neq j$, the following isomorphism of $K$-algebras holds:
\begin{equation}
\End(X)/\Rad(X) \cong \bigoplus_{j \in J} \Mat(r_j,\topp(X_j)).
\end{equation}
\begin{equation}
    \phi \longrightarrow (\overline{\phi}_j)_{j \in J}
\end{equation}

Proposition \ref{rad} can therefore be rephrased as: $\phi$ is an isomorphism if and only if $\overline{\phi_j}$ is invertible for each $j \in J$.

\end{oss}

\subsubsection{Indecomposable over finite fields and Kac polynomials}
\label{finitefieldkac}

In this paragraph, unless explicitly specified, we assume $K=\mathbb{F}_q$. The Wedderburn's theorem implies that every finite dimensional division algebra over $\mathbb{F}_q$ is a finite field. For an indecomposable representation $M\in\Rep_{\mathbb{F}_q}(Q)$, we have therefore $\topp(M)=\mathbb{F}_{q^d}$ for some $d \geq 1$.

Fix a representation $X$ with Krull-Schimdt's decomposition $X=\bigoplus_{j \in J}X_j^{r_j} $ and integers $d_j$ such that $\topp(X_j)=\mathbb{F}_{q^{d_j}}$ for each $j \in J$. From Remark \ref{isomff}, there is a  morphism of finite groups \begin{equation}p_X:\Aut(X) \to \prod_{j \in J} \Gl_{r_j}(\mathbb{F}_{q^{d_j}}) \end{equation} \begin{equation}\phi \longrightarrow (\overline{\phi})_{j \in J} . \end{equation}
and its kernel $\Ker(p_X)$ is the subset $U_X \coloneqq \{1+f \ | \ f \in \Rad(X)\} \subseteq \Aut(X)$. In particular, all the elements inside $\Ker(p_X)$ are unipotent. 

The integers $d_j$ admit the following description in terms of \textit{absolutely indecomposable} representations (see Definition \ref{abs} below).

\begin{definizione}
\label{abs}

A representation $V$ over a field $K$ is absolutely indecomposable if $V \otimes_{K} \overline{K}$ is indecomposable.
\end{definizione}

To relate absolutely indecomposable and indecomposable representations over $\mathbb{F}_q $, we introduce the action of the Frobenius. Denote by $(\overline{\mathbb{F}}_q)_{\Fr^i}$ the set $\overline{\mathbb{F}}_q$ with the structrure of $\overline{\mathbb{F}}_q$ vector space given by $$\lambda \cdot v=\Fr^i(\lambda)v .$$ Consider an $\overline{\mathbb{F}}_q$-representation $N$. We define $\Fr^i(N)$ to be the representation $N \otimes_{\overline{\mathbb{F}}_q}(\overline{\mathbb{F}}_q)_{\Fr^i}.$ 
Over the parameter space $R(Q,\alpha)(\overline{\mathbb{F}}_q)$, the Frobenius action corresponds to the usual (geometric) Frobenius, i.e for $x \in R(Q,\alpha)(\overline{\mathbb{F}}_q)$ the representation $\Fr^i(x)$ is given by the element $F^i(x) \in R(Q,\alpha)(\overline{\mathbb{F}}_q)$, where $F:R(Q,\alpha) \to R(Q,\alpha)$ raises each entry of each matrix to its $q$-th power.

\vspace{8 pt}

Fix now an indecomposable representation $M$ over $\mathbb{F}_q$ and let $d \in \N$ be such that $\topp(M)=\mathbb{F}_{q^d}$. In \cite{kacconj} it is shown that there exists  an absolutely indecomposable representation $W$ of $Q$ over $\mathbb{F}_{q^d}$ such that $\dim(W)=\frac{\alpha}{d}$ and the following isomorphism holds :\begin{equation}
\label{rad2} M \otimes_{\mathbb{F}_q} \mathbb{F}_{q^d} \cong \bigoplus_{i=0}^{d-1} \Fr^i(W). \end{equation}

\vspace{8 pt}

By a theorem of Kac (see \cite[Theorem A]{kacconj}),  there exists a polynomial with integer coefficients $a_{Q,\alpha}(t) \in \Z[t]$ such that, for each $q$,  $a_{Q,\alpha}(q)$ is equal to the number of isomorphism classes of absolutely indecomposable representations of $Q$ over $\mathbb{F}_q$ of dimension vector $\alpha$.

Thanks to the isomorphism (\ref{rad2}), in \cite[Section 1.14]{kacconj} it is shown that the number of isomorphism classes of indecomposable representations $M$ such that $\dim M=\alpha$ and $\topp(M)=\mathbb{F}_{q^d}$ is given by \begin{equation}
    \label{rad3}
    \dfrac{1}{d}\sum_{r | d}\mu\left(\dfrac{d}{r}\right)a_{Q,\frac{\alpha}{d}}(q^r)
\end{equation}

 Moreover, Kac \cite[Theorem A,Theorem B]{kacconj} gave the following combinatorial characterization of the dimension vectors $\alpha$ such that  $a_{Q,\alpha}(t)$ does not vanish. Let $\Phi(Q) \subseteq \Z^I$ be the root system associated to $Q$ as defined by Kac \cite[Section (a)]{kacconj} and let $\Phi(Q)^+  \subseteq \N^I \cap \Phi(Q)$ the subset of the positive roots of the former root system.

\begin{teorema}
\label{comb}
The Kac polynomial $a_{Q,\alpha}(t) \neq 0$ if and only if $\alpha \in \Phi(Q)^+.$ Moreover, $a_{Q,\alpha}(t) \equiv 1$ if and only if $\alpha$ is a real root (see \cite[Section (a)]{kacconj} for a definition). Otherwise $a_{Q,\alpha}(t)$ is monic of degree $\dfrac{2-\alpha^tC\alpha}{2}$. 
\end{teorema}
Here $C=(C_{i,j})_{i,j \in I}$ is the Cartan matrix of the quiver given by $$
C_{ij}=\begin{cases} 2-2(\text{the number of edges joining $i$ to
    itself})\hspace{.2cm}\text{if }i=j\\ 
  - (\text{the number of edges joining $i$ to
    $j$})\hspace{1.2cm}\text{ otherwise}.
         \end{cases}
$$

We end by recalling the Kac conjecture which was proved by Hausel, Letellier, Rodriguez-Villegas in \cite[Corollary 1.5]{aha3} 
\begin{teorema}
\label{kacconj}
For any $\alpha \in \N^I$, the Kac polynomial $a_{Q.\alpha}(t)$ has nonnegative integer coefficients.
\end{teorema}

\section{Quiver representations  of level $V$}
\label{ziopad}
\subsection{Generating functions and plethystic identities} 
Let $Q=(I,\Omega)$ be a finite quiver and let $V$ be a subset of $\N^I$. For $\alpha \in \N^I$  we denote by $\mathbb{N}^I_{\leq \alpha}$ the subset $\mathbb{N}^I_{\leq \alpha}\coloneqq\{0 \leq \beta \leq \alpha \ | \ \beta \in \N^I\}$ and similarly $V_{\leq \alpha} \coloneqq\{0 \leq \beta \leq \alpha \ | \ \beta \in V\}$. 

\begin{esempio}
Given $\lambda \in \C^I$, we denote by $V_{\lambda}$ the subset  $V_{\lambda}\coloneqq\{\beta \in \N^I \ | \ \beta \cdot \lambda=0\}$, where $\cdot$ is the canonical orthogonal product on $\N^I$. Notice that for $\lambda=0$ we have $V_{\lambda}=\N^I$.

\end{esempio}

\vspace{8 pt}

To a representation $X$ of dimension $\alpha$, we associate the following subset $\mathcal{H}_{X} \subseteq \N^I_{\leq \alpha}$. Given the decomposition of $X \otimes_K \overline{K}$ into indecomposable components $$X \otimes_K \overline{K}=\bigoplus_{j \in J} Y_j^{r_j} ,$$ we define
\begin{equation}
    \label{definitionsethquiver}
    \mathcal{H}_X\coloneqq \{0 < \beta \leq \alpha \ | \ \exists j \text{ s.t } \beta= \dim Y_j\}. 
\end{equation}

For any $V \subseteq \N^I$, we give the following definition of the representations of the quiver $Q$ of level $V$.

\begin{definizione}
\label{W-generic}
A representation $X$ of dimension $\alpha$ is said to be of level $V$ if   we have $\mathcal{H}_X=V_{\leq \alpha}$. For $V=V_{\lambda}$ with $\lambda \in \C^I$, we say that $X$ is of level $\lambda$.
\end{definizione}

\vspace{8 pt}

\begin{esempio}
 Let $V=\{\alpha\}$ for a vector $\alpha \in \N^I$. A representation $X \in R(Q,\alpha)$ is of level $\{\alpha\}$ if and only if $X \otimes_K \overline{K}$ is indecomposable, i.e if and only if $X$ is absolutely indecomposable.
\end{esempio}

\vspace{8 pt}

The notion of being of level $V$ induces a partition (indexed by the subsets of $\N^I_{\leq \alpha}$) on the representations of $Q$ of dimension $\alpha$. As on the set of  subsets of $\N^I_{\leq \alpha}$ there is a natural order relation, induced by inclusion, we can also consider the filtration associated to such a partition. In particular, we give the following definition of a representation of level at most $V$.

\begin{definizione}
\label{representationsoflevelatmostV}
For a subset $V \subseteq \N^I$, a representation $X$ is said to be of level at most $V$ if it is of level $V'$ for some $V' \subseteq V$. For $V=V_{\lambda}$ we say that a representation is of level at most $\lambda$.
\end{definizione}

\begin{oss}
Notice that a representation $X$ of dimension $\alpha$ is of level at most $V$ if and only if $\mathcal{H}_X \subseteq V_{\leq \alpha}$. In particular,
 for $V=\N^I$, any representation of $Q$ is of level at most $\N^I$.
\end{oss}

\vspace{10 pt}

If $K=\mathbb{F}_q$ and $\lambda \in \C^I$, the  Definition \ref{representationsoflevelatmostV} for representations of level at most $\lambda$ is equivalent to the following one:

\begin{lemma}
\label{def}
 A representation $X$ is of level at most $\lambda$ if given its Krull-Schimdt decomposition $\displaystyle X=\bigoplus_{j \in J}X_j^{r_j}$ we have $\dim X_j \cdot \lambda =0$ for each $j \in J$.
\end{lemma}

\begin{proof}
\label{lambda}
Given the Krull-Schidmt decomposition $\displaystyle X=\bigoplus_{j \in J}X_j^{r_j}$, for each $j \in J$  there is an isomorphism $X_j \otimes_{\mathbb{F}_q} \overline{\mathbb{F}_q}=Y_j \oplus F(Y_j) \oplus \cdots \oplus F^{d_j-1}(Y_j) $ where $\topp(X_j)=\mathbb{F}_{q^{d_j}} $. The decomposition in indecomposable factors of $X \otimes_{\mathbb{F}_q} \overline{\mathbb{F}}_q$ is thus $$\displaystyle X \otimes_{\mathbb{F}_q}\overline{\mathbb{F}_q}=  \bigoplus _{j \in J}(Y_j \oplus F(Y_j) \oplus \cdots \oplus F^{d_j-1}(Y_j))$$ with  $\dim Y_j=\frac{\dim x_j}{d_j}$. We have then $\dim Y_j \cdot \lambda=0$ if and only if $\dim X_j \cdot \lambda=0$, for each $j \in J$.
\end{proof}

\vspace{10 pt}

For $\alpha \in \N^I$, the representations of $Q$ of level at most $V$ form a constructible subset of $R(Q,\alpha)$, as explained by the following proposition:
 
 \begin{prop}
 \label{constructile}
 There exists a constructible subset $R(Q,\alpha,V)\subseteq R(Q,\alpha)$ such that ,for each extension $K \subseteq L$, the set of $L$-points $R(Q,\alpha,V)(L)$ is the subset of  representations of level at most $V$ of $Q$ over $L$. If $V=V_{\lambda}$ for  $\lambda \in \C^I$, we denote $R(Q,\alpha,V_{\lambda})$ by $R(Q,\alpha,\lambda)$.
 \end{prop}

\begin{proof}
In \cite[Section 1,8]{kacconj}, Kac showed that, for any $\alpha \in \N^I$, there exists a 
 constructible subset $A(Q,\alpha) \subseteq R(Q,\alpha)$, such that, for any field extension $K \subseteq L$, the set of $L$-points $A(Q,\alpha)(L)$ is the subset of absolutely indecomposable representations over $L$.
 
 Let $\Psi_{V,\alpha}$ be the constructible subset defined by $$\Psi_{V,\alpha}=\bigoplus_{\substack{\alpha_1,\dots,\alpha_r \in V \\ \text{s.t } \alpha_1+\cdots +\alpha_r=\alpha}} A(Q,\alpha_1) \times \cdots \times A(Q,\alpha_r) \subseteq  \bigoplus_{\substack{\alpha_1,\dots,\alpha_r \in V \\ \text{s.t } \alpha_1+\cdots +\alpha_r=\alpha}} R(Q,\alpha_1) \times \cdots \times R(Q,\alpha_r).$$

Consider the morphism $$\Phi_{V,\alpha}:\Gl_{\alpha} \times \Psi_{V,\alpha} \to R(Q,\alpha)$$ defined  by $$\Phi_{V,\alpha}(g,M_1,\dots,M_r)\coloneqq g \cdot (M_1 \oplus \cdots \oplus M_r).$$ for $g \in \Gl_{\alpha}$, $M_1 \in A(Q,\alpha_1), \dots, M_r \in A(Q,\alpha_r)$ and $\alpha_1+\cdots +\alpha_r=\alpha$. We define $R(Q,\alpha,V)$ to be the image of $\Phi_{V,\alpha}$, which is a constructible subset by Chevalley's theorem.

\end{proof}

\vspace{8 pt}

\begin{esempio}
\label{Jordanquiverexample}

Consider the case where $Q$ is the Jordan quiver (the quiver with one vertex and one arrow), i.e $|I|=|\Omega|=1$. For $n \in \N$ and a field $K$, the representation space $R(Q,n)$ is given by the $n \times n$ matrices $\Mat(n,K)$. The isomorphism classes of $R(Q,n)$ correspond to the conjugacy classes of $\Mat(n,K)$.

If $\overline{K}=K$, the indecomposable representations correspond to matrices conjugated to a single Jordan block and the decomposition into indecomposable components of a representation $M \in \Mat(n,K)$ corresponds to the writing of $M$ into its Jordan form.

Consider now the subset $V=\{1\} \subseteq \N$. Notice that in this case, a representation is of level $\{1\}$ if and only if is of level at most $\{1\}$.

A matrix $M \in \Mat(n,K)$ is of level $\{1\}$ if and only if its Jordan form over $\overline{K}$ has only blocks of size $1$, i.e if and only if $M$ is diagonalizable over $\overline{K}$.

The subset $R(Q,n,\{1\}) \subseteq R(Q,n)=\Mat(n,K)$ is therefore given by the semisimple matrices of size $n$.     
\end{esempio}

\vspace{10 pt}

\begin{oss}
\label{Lemma1}

Let $k$ be an algebraically closed field of $\ch=0$ and let $\lambda$ be an element of $ k^I$. We denote by $\pi$ the projection map $\pi_{\lambda}:\mu_{\alpha}^{-1}(\lambda) \to R(Q,\alpha)$ sending $(x_a,x_{a^*})_{a \in \Omega}$ to $\pi((x_a,X_{a^*})_{a \in \Omega})=(X_a)_{a \in \Omega}$.

From the result of Crawley-Boevey \cite[Lemma 3.2]{CrB}, we deduce that a representation $x \in R(Q,\alpha) $ belongs to $\Imm(\pi_{\lambda})$ if and only if, given its Krull-Schimdt decomposition $x=\bigoplus_{j \in J}x_j$, we have $\dim x_j \cdot \lambda=0$ for each $j \in J$. For $K=\C$, we have therefore an equality $R(Q,\alpha,\lambda)=\Imm(\pi_{\lambda})$. 

\end{oss}

Altough the result of Crawley-Boevey is stated there only for algebraically closed field of $\ch=0$, the Remark above can be extended  over a finite field $\mathbb{F}_q$ of sufficiently big characteristic. More precisely, consider $\lambda \in \Z^I$ and still denote by $\lambda$ the corresponding element  of $\F_q^I$.

Using an argument similar to the one of  Proof of Lemma \ref{def}, we deduce that, if $q >>0$, we still have $\Imm(\pi_{\lambda})=R(Q,\alpha,\lambda)$.

\vspace{12 pt}

 The following Lemma provides a way to compute the number of isomorphism classes of  representations of level at most $V$ of dimension $\alpha$ over a finite field $\mathbb{F}_q$ .

\begin{lemma}
 \label{formulaV}
 For each $V \subseteq \N^I$ and $\alpha \in \N ^I$ there exists a polynomial $M_{Q,\alpha,V}(t) \in \Z[t]$ such that, for any $q$, $M_{Q,\alpha,V}(q)$ is equal to the number of isomorphism classes of  representations of level at most $V$ of dimension vector $\alpha$ over $\mathbb{F}_q$. Moreover, the following identity holds:   \begin{equation}
    \label{Hua4}
    \Plexp\left(\sum_{\beta \in V}a_{Q,\beta}(t)y^{\beta}\right)=\sum_{\alpha \in \N^I}M_{Q,\alpha,V}(t)y^{\alpha}
\end{equation}

\end{lemma}

\begin{proof}

For $\beta \in \N^I$ denote by $a_{Q,\beta,V}(t)$ the polynomial defined by:

$$a_{Q,\beta,V}(t)=\begin{cases}
0 \ \ \text{if } \beta \not\in V\\ 
a_{Q,\beta}(t) \ \ \text{if } \beta \in V \end{cases}
.$$

The proof of the Formula (\ref{rad3}) given by Kac \cite{kacconj} can be slightly modified to show that the number of isomorphism classes of  indecomposable representations $M$ of level at most $V$ of dimension $\beta$ over $\mathbb{F}_q$ such that $\topp(M)=\mathbb{F}_{q^d}$ is equal to $$\displaystyle \dfrac{1}{d}\sum_{r |d}\mu\left(\dfrac{d}{r}\right)a_{Q_V,\frac{\alpha}{d}}(q^r) .$$ Let then $I_{Q,\beta,V}(t)$ be the polynomial defined by:

\begin{equation}
\label{Hua8}
I_{Q,\alpha,V}(t)=\sum_{d | \overline{\alpha}}\dfrac{1}{d}\sum_{r |d}\mu\left(\dfrac{d}{r}\right)a_{Q,\frac{\alpha}{d},V}(t^r)
\end{equation}
 
Notice that for any $q$, $I_{Q,\beta,V}(q)$ is equal to the number of isomorphism classes of  indecomposable representations of level at most $V$ and of dimension $\beta$ over $\mathbb{F}_q$. For each $\gamma \in \N^I$, denote by $M_{Q,\gamma,V}(t)$ the polynomials defined by the following identity:
\begin{equation}
\label{Hua7}
\sum_{\gamma \in \N^I}M_{Q,\gamma,V}(t)y^{\gamma}=\prod_{\beta \in \N^I}\exp(-\log(1-y^{\beta})I_{Q,\beta,V}(t)).
\end{equation}

As $\Rep_{\mathbb{F}_q}(Q)$ is a Krull-Schmidt category, we deduce that, for any $q$, $M_{Q,\gamma,V}(q)$ is equal to the number of isomorphism classes of  representations of level at most $V$ of dimension $\gamma$ over $\mathbb{F}_q$. Specializing the identity (\ref{Hua7}) at $t=q$ we can rewrite it as the following identity : \begin{equation}
   \label{Hua9}
   \sum_{\alpha \in \N^I}M_{Q,\alpha,V}(q)y^{\alpha}=\prod_{\alpha \in \N^I}\prod_{\substack {M \in Ind(Q,\alpha,V)/\cong \\ \topp(M)=\mathbb{F}_{q^d}}}\dfrac{1}{1-y^{\frac{\alpha}{d}(d)}}=\prod_{\alpha \in \N^I}\prod_{d | \overline{\alpha}}\psi_d\left(\dfrac{1}{1-y^{\frac{\alpha}{d}}}\right)^{\frac{1}{d}\sum_{r |d}\mu\left(\frac{d}{r}\right)a_{Q,\frac{\alpha}{d},V}(q^r)} 
    \end{equation} 
 
The right hand side of Equation (\ref{Hua9}) can  be rewritten as \begin{equation}
    \label{Hua10}
    \prod_{\alpha \in \N^I}\prod_{d | \overline{\alpha}}\psi_d\left(\dfrac{1}{1-y^{\frac{\alpha}{d}}}\right)^{\frac{1}{d}\sum_{r |d}\mu\left(\frac{d}{r}\right)a_{Q_V,\frac{\alpha}{d}}(q^r)}=\prod_{\gamma \in \N^I}\prod_{d \geq 1}\psi_d\left(\dfrac{1}{1-y^{\gamma}}\right)^{\frac{1}{d}\sum_{r |d}\mu\left(\frac{d}{r}\right)a_{Q,\gamma,V}(q^r)}
\end{equation} 

As Equation (\ref{Hua10}) holds for any $q$, we deduce that the following identity holds \begin{equation}
    \label{Hua11}
    \log\left(\sum_{\alpha \in \N^I}M_{Q,\alpha,V}(t)y^{\alpha}\right)=\sum_{\substack{\gamma \in \N^I \\ d \geq 1}}\psi_d\left(\dfrac{1}{1-y^{\gamma}}\right)(a_{Q,\gamma,V}(t))_d
\end{equation}
where $\displaystyle (a_{Q,\gamma,V}(t))_d=\dfrac{1}{d}\sum_{r| d}\mu\left(\frac{d}{r}\right)a_{Q,\gamma,V}(t^r)$. By Lemma \ref{exp} and Equation (\ref{Hua11}), we deduce finally: 
 \begin{equation}
    \Plelog\left(\sum_{\alpha \in \N^I}M_{Q,\alpha,V}(t)y^{\alpha}\right)=\sum_{\alpha \in \N^I}a_{Q,\alpha,V}(t)\Plelog\left(\dfrac{1}{1-y^{\alpha}}\right)=\sum_{\alpha \in V}a_{Q,\alpha}(t)y^{\alpha}.
\end{equation}

\end{proof}

\vspace{11 pt}

\vspace{8 pt}

\begin{esempio}

Consider the Jordan quiver $Q$ of Example  \ref{Jordanquiverexample}, $V=\{1\}$ and apply Formula (\ref{Hua4}). As $a_{Q,1}(t)=t$,  we find 
\begin{equation}
\label{Jordanquiverexample1}
\sum_{n \in \N}M_{Q,n,\{1\}}(t)y^n=\Plexp\left(ty\right)=\sum_{n \in \N}t^ny^n
\end{equation}
where the last equality comes the computation of Eq.(\ref{esempio}). We obtain therefore $$M_{Q,n,\{1\}}(t)=t^n .$$  Evaluating   $M_{Q,n,\{1\}}(t)$ at $t=q$, by Example \ref{Jordanquiverexample}, we find that the number of semisimple conjugacy classes of $\Mat(n,\F_q)$ is $q^n$. This is a classical combinatorial result (see for example \cite{Kundu}) which comes from the observation that the semisimple conjugacy classes are in bijection with the monic polynomials of $\F_q[t]$ of degree $n$.
\end{esempio}

\vspace{8 pt}

From  Lemma \ref{formulaV}, we deduce the following proposition.
\begin{prop}
\label{nonnega1}
\noindent (1)
For any subset $V \subseteq \N^I$ and $\alpha \in \N^I$, the polynomial $M_{Q,\alpha,V}(t)$ has nonnegative integers coefficients.

\noindent (2)
The polynomial $M_{Q,\alpha,V}(t)$ is non-zero if and only there exist
\begin{itemize}
    \item $\beta_1,\dots, \beta_r \in (\Phi^+(Q) \cap V)$
    \item $h_1,\dots ,h_r \in \N$
\end{itemize}
such that $h_1\beta_1+\cdots +h_r\beta_r=\alpha$

\end{prop}

\begin{proof}
 By Lemma \ref{Z} the polynomials $M_{Q,\alpha,V}(t) $ have integer coefficients. By the Definition of $\Plexp$ and Lemma \ref{formulaV}, $M_{Q,\alpha,V}(t)$ is a sum of products of the form 
 \begin{equation}
     \label{posprop}
 \frac{a_{Q,\beta_1}(t^{n_1})^{m_1}}{k_1}\frac{a_{Q,\beta_2}(t^{n_2})^{m_2}}{k_2}\cdots \frac{a_{Q,\beta_l}(t^{n_l})^{m_l}}{k_l}
 \end{equation}
 with $k_1,\dots k_l,m_1,\dots,m_l,n_1,\dots,n_l$ positive integers such that $m_1 n_1 \beta_1+\cdots +n_l m_l \beta_l=\alpha $. Kac conjecture (see Theorem \ref{kacconj}) implies that these products have nonnegative coefficients. By Proposition \ref{comb}, we see that a product as in Equation (\ref{posprop}) is different from $0$ if and only if $\beta_1,\dots ,\beta_l \in \Phi^+(Q)$ and so we deduce property $(2)$.   
\end{proof}

\subsubsection{Quiver stacks}

Let $K=\C$ and fix $\lambda \in \C^I$. Let $\alpha \in \N^I$ and consider the associated quiver stack $\mathcal{M}_{\lambda,\alpha}$. For each $i \in \Z$, let $H^i_c(\mathcal{M}_{\lambda,\alpha})\coloneqq H^i_c(\mathcal{M}_{\lambda,\alpha},\C)$ be the compactly supported cohomology of the quotient stack $\mathcal{
M}_{\lambda,\alpha}$ with coefficients in $\C.$ We denote by $P_c(\mathcal{M}_{\lambda,\alpha},t)$ the compactly supported Poincar\'e series $$\displaystyle P_c(\mathcal{M}_{\lambda,\alpha},t)\coloneqq \sum_{i \in \Z} \dim H^{i}_c(\mathcal{M}_{\lambda,\alpha})t^i .$$ Davison \cite[Theorem B]{Bro} showed that the following equality holds: \begin{equation}
    \label{delignesimp}
    P_c(\mathcal{M}_{\lambda,\alpha},t)=t^{-2(\alpha,\alpha)}\Coeff_{\alpha}\left(\Plexp\left(\sum_{\beta \cdot \lambda=0}\dfrac{t^2}{t^2-1}a_{Q,\beta}(t^2)y^{\beta}\right)\right)
\end{equation}

where $(\alpha,\alpha)$ is the Euler form of $Q$. From  Identity (\ref{delignesimp}) and Proposition \ref{formulaV} in the case where $V=V_{\lambda}$, we deduce the following Proposition.

\begin{prop}
\label{propdelignesimp}
\begin{itemize}
    \item The quiver stack $\mathcal{M}_{\lambda,\alpha}$ (and so the quiver variety $Q_{\lambda,\alpha}$) is non empty if and only $M_{Q,\alpha,\lambda}(t) \neq 0$.
    \item The number of irreducible components of top dimension of the stack $\mathcal{M}_{\lambda,\alpha}$ is equal to the top degree coefficient of $M_{Q,\alpha,\lambda}(t)$.
\end{itemize}
\end{prop}

\subsection{Counting  representations of level at most $V$}
\label{countingV}

For $\alpha \in \N^I$, we denote by $\Gl_{\alpha}(\F_q)$ the finite group $\Gl_{\alpha}(\F_q)\coloneqq \displaystyle \prod_{i \in I}\Gl_{\alpha_i}(\F_q)$. In the paper \cite[Theorem 1.1]{letellierDT}, Letellier linked the Kac polynomial $a_{Q,\alpha}(t)$ to the representation theory of the finite group $\Gl_{\alpha}(\F_q)$. In this paragraph, we will explain how to generalize his results to the case of  representations of level at most $V$ for certain subsets $V \subseteq \N^I$.

\vspace{8 pt}

The finite group $\Gl_{\alpha}(\F_q)$ acts on the finite set $R(Q,\alpha)(\mathbb{F}_q)$. We denote the associated complex character of $\Gl_{\alpha}(\F_q)$ by $R_{\alpha}$. Let $\det_I$ be the morphism $\det_I:\Gl_{\alpha}(q) \to (\F_q^*)^I$ which maps an element $(g_i)_{i \in I}$ to the element $(\det(g_i))_{i \in I}$. 

Given $\sigma=(\sigma_i)_{i \in I} \in \Hom(\F_q^*,\C^*)^I$ and $\delta \in \N^I$, we denote by $\sigma^{\delta}$ the element of $\Hom(\mathbb{F}_q^*,\C^*)$ defined by $$\sigma^{\delta}\coloneqq \prod_{i \in I}\sigma_{i}^{\delta_i} .$$ 

Given $\sigma \in \Hom(\mathbb{F}^*_q,\C^*)^I$, we  denote by $\mathcal{H}_{\sigma}$ the subset of $\N^I$ defined by $\mathcal{H}_{\sigma}\coloneqq\{\delta \in \N^I \ | \ \sigma^{\delta}=1\}$ and by  $\mathcal{H}_{\sigma,\alpha}$  the intersection $\mathcal{H}_{\sigma,\alpha}\coloneqq\mathcal{H}_{\sigma} \cap \N^I_{\leq \alpha}$. Finally, we denote by $\rho$ the character of $\Gl_{\alpha}(\F_q)$ defined by $$\rho((g_i)_{i \in I})\coloneqq \prod_{i \in I} \sigma_i(\det(g_i)) .$$ 

\vspace{8 pt}
\begin{oss}
\label{detrm}
Fix a dimension vector $\beta \in \N^I$, an integer $r \in \N$, an  indecomposable representation $M \in R(Q,\beta)(\mathbb{F}_q)$ and denote by $N$ the representation $N=M^{\oplus r}$. As seen at the beginning of paragraph $\cref{finitefieldkac}$, there is an isomorphism $\Aut(N)/U_N \cong \Gl_{r}(\mathbb{F}_{q^d})$ where $\topp(M)=\mathbb{F}_{q^d}$. As $U_N$ is a unipotent subgroup, the morphism $\det_I$ passes to the quotient $\Aut(N)/U_N$ and induces thus a morphism  $\det_I:\Gl_{r}(\mathbb{F}_{q^d})\to (\mathbb{F}^*_q)^I$. Its value at a matrix $A \in \Gl_{r}(\mathbb{F}_{q^d})$ is given by 
\begin{equation}
\det{}_{I}(A)=(N_{\mathbb{F}_{q^d}/\mathbb{F}_q}(\det(A)^{\frac{\beta_i}{d}}))_{i \in I} 
\end{equation}

\end{oss}

\vspace{8 pt}
The main result of this paragraph is the following Theorem:

\begin{teorema}
\label{multrep}
For a subset $V \subseteq \N^I$ and a dimension vector $\alpha \in \N^I$, if there exists an element $\sigma \in \Hom(\mathbb{F}_q^*,\C^*)^I$ such that $V_{\leq \alpha}=\mathcal{H}_{\sigma,\alpha}$, the following equality holds:

\begin{equation}
    \label{multrep1}
    \langle R_{\alpha} \otimes \rho ,1\rangle=M_{Q,\alpha,V}(q)
\end{equation}
\end{teorema}

\begin{proof}

 Fix a representation $x$ inside $R(Q,\alpha)$: we start by showing that $ \rho|_{\Stab(x)} \cong 1$ if and only if $ x \in R(Q,\alpha,V)$. Consider the Krull-Schmidt decomposition $$x=\bigoplus_{j \in J}x_j^{r_j} .$$ Let $\beta_j$ be the dimension vector $\dim x_j$ and $d_j$ the integer such that $\mathbb{F}_{q^{d_j}}=\topp(x_j)$ for $j \in J$.
 
 As explained at the beginning of paragraph \cref{finitefieldkac}, quotienting by the subgroup  $U_x \subseteq \Stab(x)$ there is an isomorphism $\Stab(x)/U_x \cong \displaystyle \prod_{j \in J} \Gl_{r_j}(\mathbb{F}_{q^{d_j}}) $. The character $\rho$ is trivial over $U_x$ and induces therefore a character $\rho:\displaystyle \prod_{j \in J} \Gl_{r_j}(\mathbb{F}_{q^{d_j}}) \to \C^*$ which by Remark \ref{detrm} is given by $$\rho((A_j))_{j \in J}=\displaystyle \prod_{j \in J} \sigma^{\frac{\beta_j}{d_j}}(N_{\mathbb{F}_{q^{d_j}}/\mathbb{F}_q}(A_j)) .$$
 
 Therefore, we deduce that $\rho|_{\Stab(x)} \equiv 1 $ if and only if $ \dfrac{\beta_j}{d_j} \in \mathcal{H}_{\sigma,\alpha}=V_{\leq \alpha}$ for each $j \in J$. This is exactly the condition that must hold for $x$ to be of level at most $V$.

\vspace{8 pt}
From the discussion above, we deduce therefore that it holds: \begin{equation}
\label{eqcounting}
\displaystyle \langle R_{\alpha} \otimes \rho,1\rangle=\dfrac{1}{|\Gl_{\alpha}(\F_q)|}\sum_{x \in R(Q,\alpha)(\F_q)}\sum_{g \in \Stab(x)}\rho(g) =\dfrac{1}{|\Gl_{\alpha}(\F_q)|}\sum_{x\in R(Q,\alpha,V)(\F_q)}|\Stab(x)| .\end{equation} Applying the Burnside formula to the RHS of eq.(\ref{eqcounting}) we obtain thus the equality $$\langle R_{\alpha} \otimes \rho,1\rangle=M_{Q,\alpha,V}(q) .$$

\end{proof}

\section{Tensor product of characters of finite general linear group}
In this section, we recall how to construct the semisimple split characters of $\Gl_n(\F_q)$ via Harisha-Chandra induction. We begin by  fixing some notations and recalling some generalities about $\Gl_n(\F_q)$.

\subsection{Notations}
\label{introrep}

Given two $\overline{\F}_q $ varieties $X,Y$ endowed with Frobenius morphisms $F_X:X \to X$, $F_Y:Y \to Y$ respectively, we write $f:(X,F_X) \to (Y,F_Y)$ for a morphism $f:X \to Y$ commuting with the Frobenius maps.

\vspace{8 pt}

For any $n$, we denote by $\Gl_n$ the general linear group over $\overline{\F}_q$. The group $\Gl_n$ is endowed with the standard Frobenius morphism $F:\Gl_n \to \Gl_n$ where $F((a_{i,j}))=(a_{i,j}^q)$. The finite general linear group $\Gl_n(\F_q)$ is thus the finite subgroup $\Gl_n^F \subseteq \Gl_n$.  

For a $1$-parameter subgroup $\eta: \mathbb{G}_m \to \Gl_n$ the \textit{weight} $|\eta|$ is defined as the integer $|\eta | \in \Z$ such that we have \begin{align}\label{def||}
    \det (\eta):\mathbb{G}_m \to \mathbb{G}_m 
               \\      z \to z^{|\eta|} \end{align}

\vspace{8 pt}

For any integers $n,d \geq 1$, we denote by $F_d:\Gl_n^d \to \Gl_n^d$ the (twisted) Frobenius $$F_d(x_1,\dots,x_d)=(F(x_d),F(x_1),\dots,F(x_{d-1})) .$$  Notice that $$(\Gl_n^d)^{F_d}=\Gl_n(\F_{q^d}) .$$

We define more generally the \textit{weight} for an homomorphism $\eta:(\mathbb{G}^d_m,F_d) \to (\Gl_n,F)$ as the integer given by\begin{equation} 
\label{def||2}
|\eta|\coloneqq \dfrac{|\eta( \Delta)|}{d}
\end{equation}
where $\Delta: \mathbb{G}_m \to (\mathbb{G}_m)^d$ is the diagonal embedding. The weight $|\eta|$ is an integer as $\eta$ is defined over $\F_q$.

\subsection{$F$-stable Levi subgroups}
\label{typeM1}

For any $F$-stable Levi subgroup $M \subseteq \Gl_n$, there exist integers $d_1,\dots ,d_r,m_1,\dots ,m_r$ such that $m_1d_1+\cdots +m_rd_r=n$ and such that \begin{equation}\label{isoLevi}
(M,F) \cong ((\Gl_{m_1})^{d_1} \times \dots \times (\Gl_{m_r})^{d_r},F_{d_1} \times \cdots \times F_{d_r}). \end{equation} 

We say that $M$ is \textit{split} if $d_1=\cdots=d_r=1$.

The isomorphism (\ref{isoLevi}) implies that there is an isomorphism  $$M^F=M(\mathbb{F}_q) \cong \displaystyle \prod_i \Gl_{m_i}(\mathbb{F}_{q^{d_i}})$$ and an isomorphism \begin{equation}
\label{centerlevi}     
 (Z_M,F) \cong  ((\mathbb{G}_m)^{d_1} \times \dots \times (\mathbb{G}_m)^{d_r},F_{d_1} \times \cdots \times F_{d_r}). \end{equation}

 Notice that, for each $j=1,\dots,r$,  the integer $m_j$ is equal to the  weight $|i_j|$ of the embedding $i_j:(\mathbb{G}_m^{d_j},F_{d_j}) \hookrightarrow (\Gl_n,F)$ induced by the isomorphism of eq.(\ref{centerlevi}).

\subsection{Harisha-Chandra induction for split Levi subgroups}
\label{splitdef}

Fix $n \in \N$ and let $G=\Gl_n$. Let $L \subseteq G$ be the split Levi subgroup  $\Gl_{n_0} \times \cdots \times \Gl_{n_s}$ embedded diagonally into $\Gl_n$ i.e
 $$L=\begin{pmatrix}
    \Gl_{n_{s}} &0 &0 &0 &0 &\dots &0\\
    0 &\Gl_{n_{s-1}} &0 &0 &0 &\dots &0 \\
    0  &0 &\Gl_{n_{s-2}} &0 &0 &\dots &0\\
    \vdots &\vdots &\vdots &\ddots &\dots &\dots &0  \\
    0 &0 &0 &0 &0 &0 &\Gl_{n_0}
    \end{pmatrix} .$$

The finite group $L^F$ is therefore isomorphic to $$L^F = \Gl_{n_0}(\F_q) \times \cdots \times \Gl_{n_s}(\F_q) .$$

Let $P$ be the $F$-stable parabolic subgroup containing $L$ and the upper triangular matrices.

Recall that the quotient $G/P$ is identified with the variety of flags inside $\overline{\mathbb{F}}_q$: \begin{equation}\label{flagover}
G/P=\{ F_s \subseteq F_{s-1} \subseteq \cdots \subseteq F_{0}= \overline{\mathbb{F}}_q^n \ : \ \dim(F_i)=\sum_{j=i}^s n_j \}. \end{equation}

The $\mathbb{F}_q$-rational points $(G/P)^F$ are thus identied with flags of vector subspaces of $\mathbb{F}_q^n$ \begin{equation}\label{flag}(G/P)^F=G^F/P^F=\{F_s \subseteq F_{s-1} \subseteq \cdots \subseteq F_{0}= \mathbb{F}_q^n \ : \ \dim(F_i)=\sum_{j=i}^s n_j \}.  \end{equation}

Denote by $\pi_L$ the quotient map $\pi_L:P \to L$. Consider now a linear character $\gamma: L^F \to \C^*$, given by $$\displaystyle \gamma((M_0,\dots,M_s))=\prod_{j=0}^s \gamma_j(\det(M_j)) ,$$ where $M_j \in \Gl_{n_j}(\F_q)$ for $j=0,\dots,s$ and the $\gamma_j$'s are homomorphisms $\gamma_j:\mathbb{F}^*_q \to \C^*$. We will usually denote by $\gamma$ also the $s+1$-uple $\gamma=(\gamma_0,\dots,\gamma_s) \in \Hom(\F_q^*,\C^*)^{s+1}$. We define the associated character $R^G_L(\gamma)$ of $G^F$ by the following rule: \begin{equation}
\label{delignelusztig}
R^G_L(\gamma)(g)=\sum_{\substack{h \in G^F/P^F \\ h^{-1}gh \in P^F}}\gamma(\pi_L(h^{-1}gh)).
\end{equation}

The  character $R^G_L(\gamma)$ is the character of the induced representation $\Ind_{P^F}^{G^F}(\gamma(\pi_L))$. Notice that, for $h \in G^F/P^F$ and $g \in G^F$, having $h^{-1}gh \in P^F$ is equivalent to having that  $g \cdot h=h$ for the action  by  left multiplication of $G^F$ on $G^F/P^F$. Via the identification (\ref{flag}), the subset $(G^F/P^F)^g$ given by the points fixed by $g$  corresponds to $g$-stable flags.

\vspace{10 pt}

We will call the characters of $G^F$ of the form $R^G_L(\gamma)$ (not necessarily irreducible)  for some $L$ and $\gamma:L^F \to \C^*$ the Harisha-Chandra characters. If $R^G_L(\gamma)$ is irreducible, we will call it a semisimple split irreducible character.

\vspace{10 pt}

The character $R^G_L(\gamma)$  is irreducible if and only if  $\gamma_i \neq \gamma_j$ for every $i \neq j$ (see \cite[Section 3]{LSr}). In \cref{irred}, we show how this classical result can be deduced from the main Theorem of this paper.

\subsection{Generic $k$-tuples of Harisha-Chandra characters}

 Hausel, Letellier, Rodriguez-Villegas \cite[Definition 2.2.5]{AH} gave the following definition of genericity for $k$-tuples of Harisha-Chandra  characters of $\Gl_n(\F_q)$. 
\begin{definizione}
\label{generic1}
We say that the a $k$-tuple $\mathcal{X}=(\Ch_1,\dots,\Ch_k)$ is generic, where $\Ch_i=R^G_{L_i}(\delta_i)$, if for any $F$-stable Levi subgroup $M \subseteq \Gl_n$ and $g_1,\dots ,g_k \in \Gl_n(\F_q)$ such that $Z_M \subseteq g_iL_ig_i^{-1}$, the character $\Gamma_M$ of $Z_M^F$ defined as $$\Gamma_M(z)=\prod_{i=1}^k \delta_i(g_izg_i^{-1})$$ for $z \in Z_M^F$, is a generic linear character of $Z_M^F$.

By this, it is meant that $\Gamma_M|_{Z_G^F} $ is trivial and for any $F$-stable $M \subseteq M' \subsetneq \Gl_n$ the restriction $\Gamma_M|_{Z_{M'}^F}$ is non trivial.

\end{definizione}

\vspace{8 pt}

Definition (\ref{generic1}) is actually stated in \cite{AH} only for $k$-tuples  of semisimple split  irreducible characters, but it translates \textit{verbatim} to a general $k$-tuple of Harisha-Chandra characters.

\begin{esempio}
Consider a $k$-tuple of irreducible characters $(\alpha_1 \circ \det,\dots,\alpha_k \circ \det)$, with $\alpha_i\in \Hom(\F_q^*,\C^*)$. This $k$-tuple is generic if and only if the element $\alpha_1 \cdots \alpha_k$ has order $n$.
\end{esempio}

\subsection{Main result}
\label{charss}
 Fix an integer $g \geq 0$ and a $k$-tuple $(L_1,\dots,L_k)$  of (split) $F$-stable Levi subgroups of $\Gl_n$  with $$L_j=\Gl_{n_{j,0}}  \times \cdots \times \Gl_{n_{j,s_j}} .$$ Let $Q=(I,\Omega)$ be the following \textit{star-shaped quiver}:

 
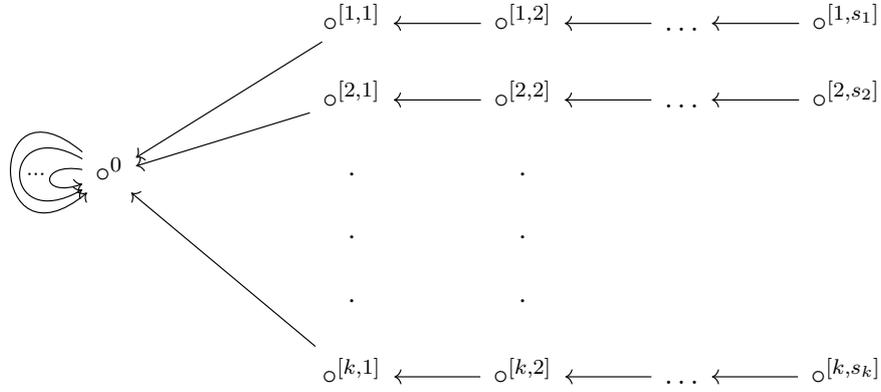
\begin{center}
    \begin{tikzcd}[row sep=1em,column sep=3em]
    & &\circ^{[1,1]} \arrow[ddll,""] &\circ^{[1,2]} \arrow[l,""]  &\dots \arrow[l,""] &\circ^{[1,s_1]} \arrow[l,""]\\
    & &\circ^{[2,1]} \arrow[dll,""] &\circ^{[2,2]} \arrow[l,""] &\dots \arrow[l,""] &\circ^{[2,s_2]} \arrow[l,""]\\
    \circ^0 \arrow[out=170,in=200,loop,swap] \arrow[out=150,in=210,loop,"\cdots"] \arrow[out=140,in=220,loop,swap]  & &\cdot &\cdot\\
    & &\cdot &\cdot\\
    & &\cdot &\cdot\\
    & &\circ^{[k,1]} \arrow[uuull,""] &\circ^{[k,2]} \arrow[l,""]  &\dots \arrow[l,""] &\circ^{[k,s_k]} \arrow[l,""]
    \end{tikzcd}
\end{center}
 
\vspace{12 pt} 
 We will denote the vertex $0$ also by $[i,0]$ for $i=1,\dots,k$. 
 Let $\alpha \in \N ^I$ be the dimension vector defined as $\alpha_0=n$ and $\alpha_{[i,j]}=n-\displaystyle\sum_{h=1}^j n_{i,h}$. 
 
 Let $(\N^I)^*\subseteq \N^I$ be the subset of dimension vectors which are non increasing along the legs. Notice that $\alpha \in (\N^I)^*$. 
 
 For any $\beta \in \N^I$, denote by  $R(Q,\beta)^*\subseteq R(Q,\beta)$ the representations which have injective maps along the legs. Notice that if $\beta \not\in (\N^I)^*$, we have $R(Q,\beta)^*=\emptyset$.
 
.

\begin{oss}
\label{indinj}
Notice that if $M \in \Rep_K(Q)$ is an indecomposable representation such that $(\dim M)_0 \neq 0$, then all the maps of $M$ along the legs are injective, for more details see for instance \cite[Lemma 3.2.1]{AH}. In particular $\dim M \in (\N^I)^*$.
\end{oss}

From Remark \ref{indinj} above, we deduce that for a subset $V \subseteq (\N^I)^*$ and any $\beta \in (\N^I)^*$, the  representations $R(Q,\beta,V)$ of level $V$  are all contained in $R(Q,\beta)^*$. In particular, for $V=(\N^I)^*$ we have an identity $R(Q,\beta,(\N^I)^*)=R(Q,\beta)^{*}$. The Formula (\ref{Hua4}) implies thus the following identity:
\begin{equation}
\label{injective}
\sum_{\beta \in (\N^I)^*}M_{Q,\beta}^*(t)y^{\alpha}=\Plexp\left(\sum_{\beta \in (\N^I)^*}a_{Q,\beta}(t)y^{\beta}\right)
\end{equation}
where the polynomials $M_{Q,\beta}^*(t)$ are such that, for any $q$, $M_{Q,\beta}^*(q)$ is equal to the number of isomorphism classes of representations of dimension $\beta$ with injective maps along the legs over $\mathbb{F}_q$.  Identity (\ref{injective}) had already been shown in \cite[Proposition 3.2.2]{AH}.

\subsubsection{Levels for $k$-tuples of characters}
\label{parameter}
Consider now, for each $i=1,\dots ,k$, a character $\gamma_i=(\gamma_{i,0},\dots ,\gamma_{i,s_i}):L_i^F \to \C^*$. To the $k$-tuple of characters $\mathcal{X}=\displaystyle (R^G_{L_i}(\gamma_i))_{i=1}^k$ we associate an element  $\sigma_{\mathcal{X}} \in \Hom(\mathbb{F}_q^*,\C^*)^I$ defined as:

\begin{equation}\label{definitionsigma}
(\sigma_{\mathcal{X}})_{[i,j]}\coloneqq\begin{cases}\displaystyle \prod_{i=1}^k \gamma_{i,0} \text{ if } j=0 \\ \gamma_{i,j}\gamma_{i,j-1}^{-1} \text{ otherwise}
\end{cases} .\end{equation}

In this context, we define the  subset $\mathcal{H}_{\sigma_{\mathcal{X}},\alpha}^* \subseteq (\N^I)^*$ as follows
$$\mathcal{H}_{\sigma_{\mathcal{X}},\alpha}^*:=\{\delta \in (\N^I)^* \ | \ 0< \delta \leq \alpha \ , \ \sigma_{\gamma}^{\delta} =  1\} .$$

\vspace{10 pt}

For a subset $V \subseteq (\N^I)^*$, we give the following definition of a $k$-tuple $(R^G_{L_1}(\gamma_1),\dots,R^G_{L_k}(\gamma_k))$ of level $V$.

\begin{definizione}
\label{defgeneric}
 The $k$-tuple $\displaystyle (R^G_{L_{1}}(\gamma_1),\dots,R^G_{L_k}(\gamma_k))$ is said to be of level $V$ if $\mathcal{H}_{\sigma_{\mathcal{X}},\alpha}^*=V_{\leq \alpha}$. For $\lambda \in \C^I$, we say that $(R^G_{L_{1}}(\gamma_1),\dots,R^G_{L_k}(\gamma_k))$ is of level $\lambda$ if it is of level $V_{\lambda}$.
\end{definizione}

\begin{oss}
Notice that any $k$-tuple $\mathcal{X}=(R^G_{L_1}(\gamma_1),\dots,R^G_{L_k}(\gamma_k))$ is automatically of level $\mathcal{H}^*_{\sigma_{\mathcal{X}}.\alpha}$.

\end{oss}

\vspace{8 pt}

\begin{esempio}
\label{esempio3}
Notice the $k$-tuple of split Levi characters $(R^G_{L_i}(1))_{i=1}^k$ is of level $(\N^I)^*$. In this case indeed $\sigma_{\mathcal{X}}=1$ and so $\mathcal{H}_{\sigma_{\mathcal{X}},\alpha}^*=(\N^I)^*_{\leq \alpha}$.
\end{esempio}
\vspace{8 pt}

We have the following Lemma for generic $k$-tuples of Harisha-Chandra characters.

\begin{lemma}
\label{generic2}
For a $k$-tuple of Harisha-Chandra characters $\mathcal{X}=(R^G_{L_{1}}(\gamma_1),\dots,R^G_{L_k}(\gamma_k))$, if  $\mathcal{H}_{\sigma_{\mathcal{X}},\alpha}^*=\{\alpha\}$, then $\Ch$ is generic as in Definition \ref{generic1}. On the other side, if the $k$-tuple $\Ch$ is generic, there are no elements $\delta, \epsilon \in \mathcal{H}_{\sigma_{\mathcal{X}},\alpha}^*\setminus \{\alpha\}$ such that $\delta+\epsilon=\alpha$.
\end{lemma}

\begin{proof}
Pick $M \subseteq G$ an $F$-stable Levi subgroup and let $m_1\dots,m_r,d_1,\dots,d_r$ be the non-negative integer associated to $M$ as in \cref{typeM1}. There is therefore an isomorphism $$(Z_M,F) \cong ((\mathbb{G}_m)^{d_1} \times \cdots \times (\mathbb{G}_m)^{d_r},F_{d_1}\times \cdots \times F_{d_r}) .$$ There exist elements $g_1,\dots ,g_k \in G^F$ such that $g_iZ_Mg_i^{-1} \subseteq L_i$ if and only if there exist $k$  embeddings  $$\lambda_i:((\mathbb{G}_m)^{d_1} \times \cdots \times (\mathbb{G}_m)^{d_r},F_{d_1}\times \cdots \times F_{d_r}) \hookrightarrow (L_i,F) $$ respecting the condition about \textit{weights} we will explicitate in Equation (\ref{condition}) below. For $j=1,\dots,r$, we denote by $\lambda_i^j:((\mathbb{G}_m)^{d_j},F_{d_j}) \to (L_i,F)$ the restriction of $\lambda_i$ to the subgroup $\{1\} \times \cdots \times (\mathbb{G}_m)^{d_j} \times \{1\} \times \cdots $ so that $\lambda_i=\displaystyle \prod_{j=1}^r \lambda_i^j$. 

The composition of $\lambda_i^j$ and the inclusion $L_i \subseteq G$ defines a morphism which we still denote by $\lambda_i^j:((\mathbb{G}_m)^{d_j},F_{d_j}) \to (G,F)$ that  must respect the following equality: \begin{equation}
\label{condition}
    |\lambda_i^j|=m_j
\end{equation}

For $i=1,\dots,k$ and $l=0,\dots,s_i$, denote by $p_{i,l}$ the projection $p_{i,l}:L_i \to \Gl_{n_{i,l}} $ and by $\lambda_{i,l}^j$ the morphism $$\lambda_{i,l}^j\coloneqq p_{i,l} \circ \lambda_i^j:((\mathbb{G}_m)^{d_j},F_{d_j}) \to (\Gl_{n_{i,l}},F) .$$ 

Denote by $\gamma_i^{g_i}:Z_M^F \to \C^*$ the morphism given by $\gamma_i^{g_i}(z)=\gamma_i(g_izg_i^{-1})$.
Via the identifications above, the character $\gamma_i^{g_i}$ corresponds  to the character $$\gamma_i \circ \lambda_i:((\mathbb{G}_m)^{d_1})^{F_{d_1}} \times \cdots \times ((\mathbb{G}_m)^{d_r})^{F_{d_r}}=\mathbb{F}^*_{q^{d_1}} \times \dots \times \mathbb{F}^*_{q^{d_r}} \to \C^* $$ given by
 \begin{equation}\label{generic3}
(x_1,\dots ,x_r)\longrightarrow \prod_{j,l}\gamma_{i,l}(N_{\mathbb{F}_{q^{d_j}}/\mathbb{F}_q}(x_j))^{|\lambda_{i,l}^j|} \end{equation}

The equality $1=\displaystyle \prod_{i=1}^k \gamma_i^{g_i}:Z_M^F \to \C^*$ therefore holds if and only if for every $j=1,\dots,r$ \begin{equation}
    \label{generic4}
   \prod_{i,l}\gamma_{i,l}(N_{\mathbb{F}_{q^{d_j}}/\mathbb{F}_q}(x_j))^{|\lambda_{i,l}^j|}=1
\end{equation}

Put $N_{\mathbb{F}_{q^{d_j}}/\mathbb{F}_q}(x_j)=y$. The following equality holds: \begin{equation}
    \label{generic5}
    \prod_{i,l}\gamma_{i,l}(y)^{|\lambda_{i,l}^j|}=\left(\prod_i \gamma_{i,0}(y)\right)^{m_j}\prod_{i,l}(\gamma_{i,l}\gamma_{i,l-1}^{-1}(y))^{m_j-\sum_{s=1}^{l-1}|\lambda_{i,s}|}=\sigma_{\mathcal{X}}^{\delta_j}(y)
    \end{equation}
where $\delta_j $ is the element of $\N^I$ given by  $(\delta_j)_0=m_j$ and $$(\delta_j)_{[i,l]}=m_j-\sum_{s=1}^{l-1}|\lambda_{i,s}|.$$ Therefore, from Equation (\ref{generic5}), we deduce that if $\mathcal{H}_{\sigma_{\mathcal{X}},\alpha}^*=\{\alpha\}$ the $k$-tuple of characters $\mathcal{X}=(R^G_{L_1}(\gamma_1),\dots,R^G_{L_k}(\gamma_k))$ is generic. 

\vspace{8 pt}

Conversely, assume that the $k$-tuple
$\mathcal{X}$ is generic 
and assume the existence of $\delta \in \mathcal{H}_{\sigma_{\mathcal{X}},\alpha}^*$ such that $\delta \neq \alpha$ and $\epsilon\coloneqq\alpha -\delta $ belongs to $\mathcal{H}_{\sigma_{\mathcal{X}},\alpha}^*$ too. Consider the $F$-stable split Levi subgroup $M= \Gl_{\delta_0} \times \Gl_{\epsilon_0}$ embedded block diagonally. Notice that in particular $Z_M \cong \mathbb{G}_m \times \mathbb{G}_m$.

For each $i=1,\dots,k$, there exist  embeddings $\lambda_{i}^1,\lambda_i^2:(\mathbb{G}_m,F) \to (L_i,F)$ such that, with the notations used before, $$|\lambda_{i,l}^1|=\delta_{[i,l]}-\delta_{[i,l+1]}$$ and $$|\lambda_{i,l}^2|=\epsilon_{[i,l]}-\epsilon_{[i,l+1]} .$$ The associated embeddings $\lambda_{i}^1 \times \lambda_i^2 :(Z_M,F) \to (L_i,F)$ correspond to elements $g_1,\dots ,g_k \in G^F$ such that $g_iZ_Mg_i^{-1}\subseteq L_i$ and for $(x_1,x_2) \in Z_M^F=\F_q^* \times \F_q^*$ we have $$\Gamma_M(z)=\prod_{i=1}^k \gamma_i^{g_i}(x_1,x_2)=\sigma_{\Ch}^{\delta}(x_1)\sigma_{\Ch}^{\epsilon}(x_2)=1 .$$

\end{proof}

\vspace{10 pt}
Let $\Lambda$  be the character of $\Gl_n(\F_q)$ induced by the conjugation action of $\Gl_n(\F_q)$ on $\mathfrak{gl}^g_n(\F_q)$. The main result of this paper is the following Theorem. 

\begin{teorema}
\label{mainteo}
Let $V \subseteq (\N^I)^*$ and let $\mathcal{X}=(R^G_{L_{1}}(\gamma_1),\dots,R^G_{L_k}(\gamma_k))$ be a $k$-tuple of Harisha-Chandra  characters of $\Gl_n(\F_q)$ of level $V$. The following equality holds:
\begin{equation}
\label{ssemp2}    
\left<\Lambda \bigotimes_{i=1}^k R^G_{L_{i}}(\gamma_i),1\right>=M_{Q,\alpha,V}(q) \end{equation}

\end{teorema}

\vspace{10 pt}

\begin{oss}
  Notice that for any $\beta_1,\beta_2 \in \mathcal{H}^*_{\sigma_{\Ch},\alpha}$ ad any $m_1,m_2 \in \N$ such that $m_1\beta_1+m_2\beta_2 \leq \alpha$, we have $m_1\beta_1+m_2\beta_2 \in \mathcal{H}^*_{\sigma_{\Ch},\alpha}$. In particular, consider a generic $k$-tuple $\Ch=(R^G_{L_{1}}(\delta_1),\dots,R^G_{L_k}(\delta_k))$.  By Lemma \ref{generic2}, we deduce that for any $\beta_1,\dots,\beta_r \in \mathcal{H}^*_{\sigma_{\Ch},\alpha}\setminus\{\alpha\}$ and $m_1,\dots,m_r \in \N$, we have that $$m_1 \beta_1 +\cdots +m_r \beta_r \neq \alpha .$$

 In particular, from the definition of $\Plexp$ and Lemma \ref{formulaV}, we see that $M_{Q,\alpha,\mathcal{H}^*_{\sigma_{\Ch},\alpha}}(t)=a_{Q,\alpha}(t)$ and therefore Formula (\ref{ssemp2}) implies the following   Identity: \begin{equation}
    \label{ex10}    \left<\Lambda\bigotimes_{i=1}^k R^G_{L_{i}}(\delta_i),1\right>=a_{Q,\alpha}(q)\end{equation} which had already been proved in \cite[Theorem 3.4.1]{AH}.
    \end{oss}    
 
\vspace{10 pt}    
 
\begin{esempio} 
Consider the case where $V=(\N^I)^*$. As remarked in Example \ref{esempio3}, the $k$-tuple $(R^G_{L_i}(1))_{i=1}^k$ is of level $(\N^I)^*$. By eq.(\ref{ssemp2}) and eq.(\ref{injective}), we  obtain the identity \begin{equation}
\label{esempioflags}
    \langle R^G_{L_1}(1) \otimes \cdots \otimes R^G_{L_k}(1), 1 \rangle=M^{\ast}_{Q,\alpha}(q)
\end{equation}

The latter identity had already been proven in \cite[Proposition 3.2.5]{AH}. Roughly speaking, in this case  Identity (\ref{esempioflags}) comes from the fact that $\displaystyle \prod_{i=1}^kR^G_{L_i}(1)(g)$, for each $g \in \Gl_n(\F_q)$, is the number of $k$-tuple of flags of $\F_q^n$ of type $\alpha$ fixed by $g$ and the Burnside formula (see \cite[Lemma 2.1.1]{AH}).
\end{esempio}

\vspace{14 pt}

Fix now a $k$-tuple $\mathcal{X}=(R^G_{L_1}(\gamma_1),\dots,R^G_{L_k}(\gamma_k))$ and  consider the character $\rho_{\mathcal{X}}=\sigma_{\mathcal{X}}(\det_I)$ as in \cref{countingV}. We denote by $R_{\alpha}^*$ the complex character of $\Gl_{\alpha}(\F_q)$ given by its action on the finite set $R(Q,\alpha)^*(\mathbb{F}_q)$. We start by proving the following preliminary lemma:
\begin{lemma}
\label{ssemp4}
The following identity holds:
\begin{equation}
\label{ssemp5}
    \left<R_{\alpha}^* \otimes \rho_{\mathcal{X}},1\right>=\left<\Lambda\bigotimes_{i=1}^kR^G_{L_{i}}(\gamma_i),1\right> 
\end{equation}
\end{lemma}

\begin{proof}
We will show the following stronger identity which implies Identity (\ref{ssemp5}). For any $g \in \Gl_n(\F_q)$ we have the following equality: \begin{equation} \label{char}\sum_{\substack{i=1,\dots,k \\j=1,\dots s_i \\ g_{i,j}\in \Gl_{\alpha_{i,j}}(\F_q)}}(R^*_{\alpha} \otimes\rho_{\mathcal{X}})(g,g_{1,1},\dots, g_{k,1},\dots )\dfrac{1}{\displaystyle \prod_{\substack{i=1,\dots,k \\j=1,\dots s_i}}|\Gl_{\alpha_{i,j}}(\F_q)|}=\Lambda(g)\prod_{i=1}^k R^G_{L_{i}}(\gamma_i)(g) .\end{equation}

 To show Equation (\ref{char}), it is enough to prove it in the case of a single leg (i.e $k=1$). We fix a (split) Levi subgroup $L \subseteq G$ with $L \cong \Gl_{n_0} \times \cdots \times \Gl_{n_l} $. The group $L$  determines a quiver $Q'=(I',\Omega')$ with $I'=\{0,1\dots,l\}$ and $g$ loops on the vertex $0$ and the dimension vector $$\beta=(n,\beta_1,\dots,\beta_l)$$ with $\displaystyle \beta_j=\sum_{h=j}^l n_j$.
 Fix a character $\gamma=(\gamma_0,\dots,\gamma_l):L^F \to \C^*$. Denote by $\sigma_{\gamma} \in \Hom(\F_q^*,\C^*)^{I'}$ the associated element and by $\rho_{\gamma}=\sigma_{\gamma}(\det_{I'})$.  Let us show that the following equality holds: \begin{equation}
    \label{char1}
    \Lambda(g) R^G_{L}(\gamma)(g)=\sum_{\substack{j=1,\dots,l \\ g_j \in \Gl_{\beta_j}(\F_q)} }R^*_{\beta}(g,g_1,\dots ,g_l)\dfrac{\rho_{\gamma}(g,g_1,\dots ,g_{l})}{|\Gl_{\beta_1}(\F_q)|\cdots |\Gl_{\beta_{l}}(\F_q)|}
\end{equation}

Let $P$ be the unique parabolic subgroup of $\Gl_n$ containing the upper triangular matrices and $L$. Recall that by the Formula (\ref{delignelusztig}) it holds: $$R^G_{L}(\gamma)(g)=\sum_{\substack{h \in G^F/P^F \\ g \cdot h=h}}\gamma (h^{-1}gh) .$$

 The character $R^*_{\beta}$ satisfies: $$R^*_{\beta}(g,g_1,\dots ,g_l)=\#\Bigg\{f \in \mathfrak{gl}^g_n(\F_q), f_l \in \Hom^{inj}(\mathbb{F}_q^{\beta_l},\mathbb{F}_q^{\beta_{l-1}}), \dots ,$$ $$f_1 \in \Hom^{inj}(\mathbb{F}_{q}^{\beta_1},\mathbb{F}_q^n) \ \  \text{s.t} \ \ g \cdot f=f, \ gf_1g_1^{-1}=f_1,\dots, \ g_{l-1}f_lg_{l}^{-1}=f_l \Bigg\} .$$

where for $V,W$ vector spaces, we denote by  $\Hom^{inj}(V,W)$  the set of  injective linear homomorphisms. We denote by $X(g)$ the set defined by: $$X(g)\coloneqq \Bigg\{g_1\in  \Gl_{\beta_1}(\F_q) , \dots , g_l \in \Gl_{\beta_l}(\F_q), \  f_l \in \Hom^{inj}(\mathbb{F}_q^{\beta_l},\mathbb{F}_q^{\beta_{l-1}}), \dots ,$$ $$f_1 \in \Hom^{inj}(\mathbb{F}_{q}^{\beta_1},\mathbb{F}_q^n) \ \  \text{s.t}  \ \, gf_1g_1^{-1}=f_1, \dots, \ g_{l-1}f_lg_{l}^{-1}=f_l \Bigg\} $$ and for $x=(g_1,\dots,g_l,f_1,\dots,f_l) \in X(g)$ we write $\rho_{\gamma}(x)\coloneqq\rho_{\gamma}(g,g_1,\dots ,g_l)$. For a fixed $g \in \Gl_n(\F_q)$, the following equality holds then : $$\sum_{\substack{{j=1},\\ g_j \in \Gl_{\beta_j}(\F_q) }}^l \dfrac{R^*_{\beta}(g,g_1,\dots,g_l)\rho_{\gamma}(g,g_1,\dots,g_l)}{|\Gl_{\beta_1}(\F_q)| \cdots |\Gl_{\beta_l}(\F_q)|}=\Lambda(g)\sum_{x \in X(g)}\dfrac{\rho_{\gamma}(x)}{|\Gl_{\beta_1}(\F_q)| \cdots |\Gl_{\beta_l}(\F_q)|} .$$

There is a map $\psi: X(g) \rightarrow (G^F/P^F)^g$ given by $$\psi((g_1,\dots,g_l,f_1,\dots,f_l))= \Imm(f_1\cdots f_l) \subseteq \Imm(f_1\cdots f_{l-1})\subseteq \dots \subseteq \Imm(f_1) \subseteq \F_q^n .$$ To see that $\psi$ is well defined we need to start to check that the subspaces $\Imm(f_1\cdots f_l), \Imm(f_1\cdots f_{l-1}),$ $\dots ,\Imm(f_1) $ are all $g$-stable. Start with $\Imm(f_1)$. We have $gf_1=f_1g_1$ and so $g(\Imm(f_1)) \subseteq \Imm(f_1)$. For a general $j \geq 1$ we see similarly $$ gf_1 \cdots f_j=f_1g_1f_2\cdots f_j=\cdots \cdots =f_1\cdots f_jg_j .$$
Let us show that the map $\psi$ is surjective. Given a $g$-stable flag $$V_l \subseteq V_{l-1} \subseteq \cdots  \subseteq V_1 \subseteq  \mathbb{F}_q^n=hP^F ,$$ we can choose for each $j=1,\dots,l$ a  basis $\mathfrak{B}_j$ of $V_j$ such that $\mathfrak{B}_j \subseteq \mathfrak{B}_{j-1}$ as ordered sets. The choices of the  $\mathfrak{B}_j$'s define morphisms $f_j:\mathbb{F}_q^{\beta_j} \hookrightarrow \mathbb{F}_q^{\beta_{j-1}}$ such that $\Imm(f_1f_2\cdots f_j) \subseteq \mathbb{F}_{q}^n$ is $g$-stable for any $j=1,\dots,l$.

For each $j=1,\dots,l$, the automorphism $g|_{\Imm(f_1f_2\cdots f_j)}$ written in the basis $\mathfrak{B}_j$ define an element $g_j \in \Gl_{\beta_j}(\F_q)$ and the element $x_h \in X(g)$ defined by  $x_h\coloneqq (g_1,\dots ,g_l,f_1, \dots ,f_l)$ is such that $\psi(x)=V_l \subseteq V_{l-1} \subseteq \cdots \subseteq V_1 \subseteq \mathbb{F}_q^n$.

There is a free action of $\displaystyle \prod_{j=1}^l \Gl_{\beta_j}(\F_q)$ on $X(g)$ defined as $$(m_1,\dots,m_h) \cdot (g_1,\dots,g_h,f_1,\dots,f_h)\coloneqq (m_1g_1m_1^{-1},\dots, m_lg_lm_l^{-1},f_1m_1^{-1},m_1f_2m_2^{-1},\dots ,f_lm_l^{-1}) .$$ The map $\psi$ is $\displaystyle \prod_{j=1}^l \Gl_{\beta_j}(\F_q)$ invariant and, for each $h \in G^F/P^F$, the fiber $\psi^{-1}(h)$ is equal to $\displaystyle \left(\prod_{j=1}^l \Gl_{\beta_j}(q)\right)\cdot x_h$. Therefore, as  $\rho_{\gamma}((m_1,\dots,m_l) \cdot x)=\rho_{\gamma}(x)$ for any $\displaystyle (m_1,\dots,m_l) \in \prod_{j=1}^l \Gl_{\beta_j}(\F_q)$, we have
$$\sum_{x \in X(g)}\dfrac{\rho_{\gamma}(x)}{|\Gl_{\beta_1}(\F_q)| \cdots |\Gl_{\beta_l}(\F_q)|}=\sum_{h \in (G^F/P^F)^g }\rho_{\gamma}(\psi^{-1}(h)) .$$

We are thus left to show that $\rho_{\gamma}(\psi^{-1}(h))=\gamma(h^{-1}gh)$. On the one side, by evaluating $\rho_{\gamma}$ at the element $x_h \in \psi^{-1}(h)$ defined above, we see that: $$\rho_{\gamma}(\psi^{-1}(h))= ((\gamma_{l-1}^{-1}\gamma_l) (\det(g|_{V_{l}})))\cdots (\gamma_0(\det (g))) .$$

On the other side, the matrix $h^{-1}gh$ is a block upper triangular matrix:

$$h^{-1}gh=\begin{pmatrix}
    g'_l & & & & & &\\
    0 &g'_{l-1} & \\
    0 &0 &g'_{l-2}\\
    0 &0 &0 &\ast\\
    0 &0 &0 &0 &\cdots \\
    0 &0 &0 &0 &0 &g'_0
    \end{pmatrix} $$  
where $g'_l $ is $g|_{V_l}$ written in the basis $\mathfrak{B}_l$,  the matrix $\begin{pmatrix} g'_l &\ast\\
0 &g'_{l-1}
\end{pmatrix} $
is equal to $g|_{V_{l-1}}$ written in the basis $\mathfrak{B}_{l-1}$ and so on. Thus, the following identity holds : $$\gamma(h^{-1}gh)=\prod_{j=0}^l\gamma_j(\det(g'_j))=$$ $$=\gamma_l(\det(g|_{V_l}))(\gamma_{l-1}(\det(g|_{V_{l-1}}))\gamma_{l-1}^{-1}(\det(g|_{V_{l}})))(\gamma_{l-2}(\det(g|_{V_{l-2}}))\gamma_{l-2}^{-1}(\det(g|_{V_{l-1}}))))$$ $$ \cdots (\gamma_0(\det(g))\gamma_0^{-1}(\det(g|_{V_1})))=$$ $$= (\gamma_{l-1}^{-1}\gamma_l (\det(g|_{V_{l}})))\cdots (\gamma_0(\det (g))) .$$

\end{proof}

\vspace{10 pt}

To end the proof of Theorem \ref{mainteo}, it is then enough to notice that the proof of Theorem \ref{multrep} can be slightly modified to show that for $V \subseteq (\N^I)^*$ such that $V_{\leq \alpha}=\mathcal{H}^*_{\sigma_{\mathcal{X}},\alpha}$ we have \begin{equation}
    \label{multrep2}
    \langle R^*_{\alpha} \otimes \rho_{\mathcal{X}} ,1\rangle=M_{Q,\alpha,V}(q)
\end{equation}.

\begin{oss}
During the proof  of the Lemma \ref{ssemp2}  we showed Identity (\ref{char1}), i.e for any $F$-stable Levi subgroup $L \subseteq \Gl_n$ and any character $\gamma:L^F \to \C^*$, it holds:

\begin{equation}
    R^G_L(\gamma)(g)=\sum_{\substack{j=1, \\ g_j \in \Gl_{\beta_j}(\F_q)} }^l R^*_{\beta}(g,g_1,\dots,g_l)\dfrac{\rho_{\gamma}(g,g_1,\dots ,g_{l})}{|\Gl_{\beta_1}(\F_q)|\cdots |\Gl_{\beta_{l}}(\F_q)|}.
\end{equation} A similar formula seems to not have been known before in the literature. It would be interesting to find a way to relate quiver representations to  Deligne-Lusztig characters associated to non-split Levi subgroups too.
\end{oss}.

\subsubsection{Non-vanishing of multiplicities}

From  Proposition \ref{nonnega1}, we deduce the following proposition.
\begin{prop}
\label{nonvan}
\noindent (1) For a $k$-tuple of Harisha-Chandra characters $\mathcal{X}=(R^G_{L_1}(\gamma_1),\dots,R^G_{L_k}(\gamma_k))$, the multiplicity $\langle \Lambda \otimes  R^G_{L_1}(\gamma_1) \otimes \cdots \otimes R^G_{L_k}(\gamma_k),1 \rangle$ is the evaluation at $q$ of a polynomial with non-negative coefficients.

 \noindent(2) Given $Q,\alpha$ as above the multiplicity $\left<\Lambda \otimes R^G_{L_1}(\gamma_1) \otimes \cdots \otimes R^G_{L_k}(\gamma_k) ,1\right>$ is non-zero if and only there exist
\begin{itemize}
    \item $\beta_1,\dots, \beta_r \in (\Phi^+(Q) \cap V)$
    \item $m_1,\dots ,m_r \in \N$
\end{itemize}
such that $m_1\beta_1+\cdots +m_r\beta_r=\alpha$
\end{prop}

Notice that this implies that if $\Phi^+(Q) \cap V=\emptyset$ the multiplicity is $0$. Similarly, if $ \alpha \not\in V$ we have that the multiplicity is $0$. Indeed, as $V_{\leq \alpha}=\mathcal{H}^*_{\sigma_{\mathcal{X}},\alpha}$, if $\beta_i \in V $ with $\beta_i\leq \alpha $ and $m_i \beta_i \leq \alpha$ too, we have $m_i \beta_i \in V$ and $\alpha=m_1\beta_1 +\cdots +m_r \beta_r \in V$. 
\vspace{10 pt}

\begin{oss}
Identity (\ref{ssemp2}) implies that the multiplicity $\displaystyle \left<\Lambda \bigotimes_{i=1}^k R^G_{L_{i}}(\gamma_i),1\right>$ does not depend on the characters $\gamma_1,\dots, \gamma_k$ but only on the Levi subgroups $L_1,\dots,L_k$ and the subset $\mathcal{H}^*_{\sigma_{\mathcal{X}},\alpha}$.

As $a_{Q,\beta}(t) \neq 0$ if and only if $\beta \in \Phi^+(Q)$, we deduce more precisely that the multiplicity depends only on the intersection of $\mathcal{H}^*_{\sigma_{\mathcal{X}},\alpha}$ and $ \Phi^+(Q)$.
\end{oss}

\subsection{Computations}

\subsubsection{Irreducibility for semisimple split characters}
\label{irred}
Consider the case of $g=0$ and $k=2$. Consider a split $F$-stable Levi subgroup $L \subseteq \Gl_n$ and $\gamma:L^F \to \C^*$. Let $\mathcal{X}$ be the couple of split Levi characters $\mathcal{X}=(R^G_L(\gamma),R^G_L(\gamma^{-1}))$.

As $R^G_L(\gamma^{-1})$ is the dual of $R^G_L(\gamma)$, we have $\langle R^G_L(\gamma)\otimes R^G_L(\gamma^{-1}),1 \rangle=\langle R^G_L(\gamma),R^G_L(\gamma) \rangle$. Using Theorem \ref{mainteo} we give an alternative proof of the classical result that \begin{equation}
\label{irreducibility}    
\langle R^G_L(\gamma),R^G_L(\gamma) \rangle=1
\end{equation}  if and only if $\gamma_i \neq \gamma_j$ for all $i\neq j$. Notice that the Identity (\ref{irreducibility}) holds if and only if the character $R^G_L(\gamma)$ is irreducible.

Let $L=\Gl_{n_0} \times \cdots \times \Gl_{n_l}$ and $g=0$. The associated  quiver $Q$ is thus  the following type $A$ quiver. 
\begin{center}
    \begin{tikzcd}
    \circ^{[1,l]} \arrow[r,""] &\circ^{[1,l-1]} \arrow[r,""]  &\dots \arrow[r,""] &\circ^{[1,1]} \arrow[r,""]
    &\circ^0
    &\circ^{[2,1]} \arrow[l,""] &\circ^{[2,2]} \arrow[l,""]  &\dots \arrow[l,""] &\circ^{[2,l]} \arrow[l,""] 
    \end{tikzcd}
\end{center}

The associated dimension vector $\alpha$ is $(n_l,n_l+n_{l-1},\dots,n-n_0,n,n-n_0,\dots,n_{l-1},n_l)$ and the element $\sigma_{\mathcal{X}}$ is equal to $$\sigma_{\mathcal{X}}=(\gamma_{l}\gamma_{l-1}^{-1},\gamma_{l-1}\gamma_{l-2}^{-1},\dots,\gamma_{1}\gamma_0^{-1},1,\gamma_{0}\gamma_1^{-1},\dots,\gamma_{l-1}^{-1}\gamma_{l-2},\gamma_{l}^{-1}\gamma_{l-1}).$$ As $Q$ is a Dynkin quiver of type $A$, the subset $\Phi^+(Q)$ has an explicit description and each root $\beta \in \Phi^+(Q)$ is real, i.e $a_{Q,\beta}(t)=1$ (see Proposition \ref{comb}). For $j=0,\dots, l$ and $h=0,\dots,l$ define the dimension vector  $\beta_{j,h}$ as $$(\beta_{j,h})_i\coloneqq \begin{cases}0 \text{ if }  i=[1,a]  \text{ with } a>j \text{ or } i=[2,b] \text{ with } b>h\\
1 \text{ otherwise}
\end{cases}.
$$
The set $\Phi^+(Q) \cap (\N^I)^*$ is given by $\{\beta_{j,h}\}_{j,h=0,\dots,l}$. Let us denote by $M_{j,h}$ the absolutely indecomposable representation of dimension vector $\beta_{j,h}$ over $\mathbb{F}_q$. Notice that $\sigma_{\mathcal{X}}^{\beta_{j,h}}=\gamma_{j}\gamma_h^{-1}$ and so we see that $\sigma_{\mathcal{X}}^{\beta_{i,i}}=1$ for every $i=0,\dots ,l$. The representation $$\displaystyle M=\bigoplus_{j=0}^l M_{j,j}^{\oplus n_j}$$ is thus of level at most $\mathcal{H}^*_{\sigma_{\mathcal{X}},\alpha}$ and the dimension vector of $M$ is equal to $\alpha$. If $\gamma_j=\gamma_h$ for $j \neq h$, we have $\beta_{j,h},\beta_{h,j} \in \mathcal{H}^*_{\sigma_{\mathcal{X}},\alpha}$.   The representation $$N=\left(\bigoplus_{m \neq 0,j,h}M_{m,m}^{\oplus n_m}\right) \oplus M_{j,j}^{n_j-1} \oplus M_{h,h}^{n_h-1} \oplus M_{0,0}^{n_0-1} \oplus M_{j,h} \oplus M_{h,j} $$ is therefore of level at most $\mathcal{H}^*_{\sigma_{\mathcal{X}},\alpha}$ and of dimension vector $\dim N=\alpha$. We deduce that $M_{Q,\alpha,\mathcal{H}_{\sigma_{\mathcal{X}}}}(t)=1$ if and only if $\gamma_j \neq \gamma_h$ for every $j \neq h$.

\subsubsection{Explicit computation for $n=2$}
\label{examplesplit}

Let us look at the case where $\mathcal{X}$ is a $k$-tuple of semisimple split characters of $\Gl_2(\F_q)$ with $\mathcal{X}=(R^G_T(\gamma_1),\dots,R^G_T(\gamma_k))$ where $T \subseteq \Gl_2$ is the (split) maximal torus of diagonal matrices. Each character $\gamma_i$ is thus of the form $$\gamma_i=(\delta_i,\beta_i):\mathbb{F}_q^* \times \mathbb{F}_q^* \to \C^* $$ with $\delta_i,\beta_i :\F_q^* \to \C^*$.

We fix $g=0$ in the following. Using the notations of the paragraph \cref{charss}, the associated quiver $Q=(I,\Omega)$ has thus a central vertex $0$ and $k$ other vertices $[1,1], \dots ,[k,1]$. We will denote the vertex $[i,1]$ simply by $i$ for each $i=1,\dots ,k$.

The associated dimension vector $\alpha$ is given by $\alpha_0=2$ and $\alpha_i=1$ for each $i=1,\dots ,k$. The quiver $Q$ and the dimension vector $\alpha$ for $k=4$ are depicted below.

\begin{center}
\begin{tikzcd}
 &       &1 \arrow[d,""]  &     \\
&1 \arrow[r," "] &2 &1 \arrow[l," "]\\
         &  &1 \arrow[u,""] 
\end{tikzcd}

\end{center}

 For the sake of simplicity, we assume that $\delta_i^{-1} =\beta_i$ and  $\beta_i^{2} \neq 1$, for each $i=1,\dots,k$. The element $\sigma_{\mathcal{X}}$ associated to $\mathcal{X}=(R^G_T(\gamma_i))_{i=1}^k$ is thus given by $\displaystyle (\sigma_{\mathcal{X}})_0=\prod_{i=1}^k \delta_i$ and $(\sigma_{\mathcal{X}})_i=\delta_i^{-2}$. Notice that $\sigma_{\mathcal{X}}^{\alpha}=1.$ We will explicitly verify that the following equality holds:  \begin{equation}
    \label{ex1}
    \left<\bigotimes_{i=1}^k R^G_{T}(\gamma_i),1\right>=\Coeff_{y^{\alpha}}\left(\Plexp\left(\sum_{\eta \in \mathcal{H}^*_{\sigma_{\Ch},\alpha}}a_{Q,\eta}(q)y^{\eta}\right)\right)=M_{Q, \alpha, \mathcal{H}_{\sigma_{\mathcal{X}},\alpha}}(q)
\end{equation} 

\vspace{8 pt}

Notice that if $\eta \leq \alpha$ and $\eta \in (\N^I)^*$ then either $(\eta)_0=1$ or $(\eta)_0=2$. For an element $\eta \in (\N^I)^*$ such that $\eta_0=1$, we have $a_{Q,\eta}(t)=1$. An element $\eta \in (\N^I)^*$ such that $\eta_0=1$ is identified by the subset $A_{\eta} \subseteq \{1,\dots,k\}$ defined as $A_{\eta}\coloneqq\{1 \leq i \leq k \ \  \text{s.t.}\ \  \eta_i=1 \} .$ 

For such an $\eta$, we have thus $$\sigma_{\mathcal{X}}^{\eta}=\left(\prod_{i=1}^k \delta_i\right)\prod_{j \in A_{\eta}}\delta_j^{-2}=\prod_{j \in A_{\eta}}\delta_j^{-1}\prod_{h \in A_{\eta}^c}\delta_h .$$ 

For $\eta_1,\dots \eta_r \in (\N^I)^*$ and $m_1,\dots,m_r \in \N^*$ such that $m_1\eta_1 +\cdots +m_r \eta_r=\alpha $ we have that either  $r=1$, $m=1$ and $\eta=\alpha$, either $r=2$, $m_1=m_2=1$ and   $\eta_1+\eta_2=\alpha$ with $\eta_1 \neq \eta_2$ and $(\eta_1)_0=(\eta_2)_0=1$ . The right hand side of Equation (\ref{ex1}) is thus equal to:

\begin{equation}
    \label{ex2}
    \Coeff_{y^{\alpha}}\left(\Plexp\left(\sum_{\eta \in \mathcal{H}_{\sigma_{\mathcal{X}},\alpha}^*}a_{Q,\eta}(q)y^{\eta}\right)\right)=a_{Q,\alpha}(q)+\dfrac{1}{2}\sum_{\substack{\eta \in \mathcal{H}^*_{\sigma_{\mathcal{X}},\alpha} \\ \text{s.t }  \eta_0=1}}a_{Q,\eta}(q)a_{Q,\alpha-\eta}(q)=a_{Q,\alpha}(q)+\dfrac{|\mathcal{H}_{\sigma_{\mathcal{X}},\alpha}^*-\{\alpha\}|}{2}
\end{equation}

Notice that the cardinality $|\mathcal{H}_{\sigma_{\mathcal{X}},\alpha}^*-\{\alpha\}|$ is even as the set $\mathcal{H}_{\sigma_{\mathcal{X}},\alpha}^*-\{\alpha\}$ admits the involution without fixed points which sends $\eta$ to $\alpha-\eta$.

The left hand side of Equation (\ref{ex1}) can be computed explicitly using the character table of $\Gl_2(\F_q)$, which can be found for example on \cite[Page 194]{DM}.

We have indeed four types for the conjugacy classes of $\Gl_2(\F_q)$ (see for example \cite[Paragraph 4.1]{HA} for the definition of the type of a conjugacy class):

\begin{enumerate}
    \item We say that $g$ is of type $\omega_1$ (and we write $g \sim \omega_1$) if $g$ is of the form $\begin{pmatrix}
    \lambda &0\\
    0 &\lambda
    \end{pmatrix}$ for $\lambda \in \mathbb{F}_q^*$
    
    \item We say that $g$ is of type $\omega_2$ (and we write $g \sim \omega_2$) if $g$ is conjugated to a Jordan block  $\begin{pmatrix}
    \lambda &1\\
    0 &\lambda
    \end{pmatrix}$ for $\lambda \in \mathbb{F}_q^*$. The centralizer of such a $g$ has cardinality $ q(q-1)$
    
    \item We say that $g$ is of type $\omega_3$ (and we write $g \sim \omega_3$) if $g$ is conjugated to a diagonal matrix $\begin{pmatrix}
    \lambda &0\\
    0 &\mu\end{pmatrix}$ with $\lambda \neq \mu \in \mathbb{F}_q^*$. The centralizer of such a $g$ has cardinality $(q-1)^2$
    
    \item 
    We say that $g$ is of type $\omega_4$ (and we write $g \sim \omega_4$) if $g$ is conjugated to a matrix of the form
    $\begin{pmatrix}
    0 &-1\\
    x x^q &x+x^q
    \end{pmatrix}$ with $x \neq x^q \in \mathbb{F}_{q^2}^*$. The centralizer of such a $g$  cardinality $q^2-1$
\end{enumerate}

For a semisimple split character of the form $R^G_T(\gamma)$ and $g \in \Gl_2(\F_q)$, the value $R^G_T(\gamma)(g)$ depends on the type of $g$ in the following way:
\begin{enumerate}
\item If $g \sim \omega_1$, then  $R^G_T(\gamma)(g)=(q+1)\gamma(\lambda)$

\item If  $g \sim \omega_2$, then $R^G_T(\gamma)(g)=\gamma(\lambda)$

\item If $g \sim \omega_3$ and $g$ is conjugated to $\begin{pmatrix}
    \lambda &0\\
    0 &\mu\end{pmatrix}$, then  $R^G_T(\gamma)(g)=\gamma(\lambda,\mu)+\gamma(\mu,\lambda)$

\item If $g\sim \omega_4$, then $R^G_T(\gamma)(g)=0$.

\end{enumerate}

Denote by $\mathbf{T}_2$ the set $\{\omega_1,\omega_2,\omega_3,\omega_4\}$. To compute the left hand side of Equation (\ref{ex1}), we will split the sum over the types of the conjugacy classes:

$$
    \left<\bigotimes_{i=1}^kR^G_T(\gamma_i) ,1\right>=\dfrac{1}{|\Gl_2(\F_q)|}\sum_{g \in \Gl_2(\F_q)}\prod_{i=1}^k R^G_T(\gamma_i)(g)=
\dfrac{1}{|\Gl_2(\F_q)|}\sum_{\omega \in \mathbf{T}_2} \sum_{g \sim \omega} \prod_{i=1}^k R^G_T(\gamma_i)(g) .$$
We see that $\displaystyle \dfrac{1}{|\Gl_2(\F_q)|}\sum_{g \sim \omega_1}R^G_T(\gamma_i)(g)=\dfrac{(q+1)^k(q-1)}{(q-1)^2q(q+1)}=\dfrac{(q+1)^{k-1}}{q(q-1)}$ and $\displaystyle \dfrac{1}{|\Gl_2(\F_q)|} \sum_{g \sim \omega_2}R^G_T(\gamma_i)(g)=\dfrac{(q-1)}{q(q-1)}=\dfrac{1}{q}$. In a similar manner, $$\displaystyle \dfrac{1}{|\Gl_2(\F_q)|} \sum_{g \sim \omega_3}R^G_T(\gamma_i)(g)=\dfrac{1}{2(q-1)^2}\sum_{\substack{\lambda,\mu \in \mathbb{F}_q^*\\ \lambda \neq \mu }}\prod_{i=1}^k (\delta_i(\lambda\mu^{-1})+\delta_i^{-1}(\lambda\mu^{-1}))=
$$ $$\frac{1}{2(q-1)^2}\sum_{\lambda,\mu \in \mathbb{F}_q^*}\prod_{i=1}^k (\delta_i(\lambda\mu^{-1})+\delta_i^{-1}(\lambda\mu^{-1}))-\dfrac{2^k(q-1)}{2(q-1)^2} .$$ 
As the homomorphism $\mathbb{F}_q^* \times \mathbb{F}_q^* \to \mathbb{F}_q^*$ which sends $(\lambda,\mu)$ to $\lambda \mu^{-1}$ is surjective, we can rewrite the sum above as: $$-\dfrac{2^{k-1}}{(q-1)}+\dfrac{(q-1)}{2(q-1)^2}\sum_{\epsilon \in \mathbb{F}_q^*}\sum_{B \subseteq \{1,\dots,k\}}\prod_{i \in B}\delta_i(\epsilon) \prod_{i \in B^c}\delta_i^{-1}(\epsilon)=$$ $$ -\dfrac{2^{k-1}}{(q-1)}+\dfrac{1}{2(q-1)}\sum_{\substack{\eta \in (\N^I)^* \\ \eta_0=1}}\sum_{\epsilon \in \mathbb{F}_q^*}\sigma_{\gamma}^{\eta}(\epsilon)= \dfrac{|\mathcal{H}_{\sigma_{\gamma},\alpha}^*-\{\alpha\}|}{2}-\dfrac{2^{k-1}}{(q-1)} .$$
 and so \begin{equation}
    \label{ex11} \left<\bigotimes_{i=1}^kR^G_T(\gamma_i) ,1\right>=\dfrac{(q+1)^{k-1}}{q(q-1)}+\dfrac{1}{q}+\dfrac{|\mathcal{H}^*_{\sigma_{\gamma},\alpha}-\{\alpha\}|}{2}-\dfrac{2^{k-1}}{(q-1)}=\dfrac{(q+1)^{k-1}+(q-1)-2^{k-1}q}{q(q-1)}+\dfrac{|\mathcal{H}_{\sigma_{\gamma},\alpha}-\{\alpha\}|}{2}.
 \end{equation}

Recall that the Identity (\ref{ex10})  gives an equality $a_{Q,\alpha}(q)=\left<\bigotimes_{i=1}^k R^G_T(\mu_i),1\right>$ for a generic $k$-tuple $(R^G_T(\mu_i))_{i=1}^k$. The multiplicity  $ \left<\bigotimes_{i=1}^k R^G_T(\mu_i),1\right>$ can be computed in the same way as Identity (\ref{ex11}) above and gives the identity: \begin{equation}
    a_{Q,\alpha}(q)=\dfrac{(q+1)^{k-1}+(q-1)-2^{k-1}q}{q(q-1)}.
\end{equation} 

From Identity (\ref{ex2}) and Identity (\ref{ex11}), we deduce Identity (\ref{ex1}).

\section{Conflict of interest}

The author has no conflict of interest to declare that are relevant to this article.

\end{document}